\documentclass[final,3p,times]{elsarticle}
\usepackage{amssymb}
 \usepackage{amsthm}
 \usepackage{amsmath}
 \usepackage{algorithmic}
 \usepackage{algorithm}
\usepackage{url}
\usepackage{xcolor}
\usepackage{caption}
\usepackage{tabulary}
\newcolumntype{K}[1]{>{\centering\arraybackslash}p{#1}}
\journal{Journal of Computational Physics}

\begin{document}

\begin{frontmatter}



\title{Solving variational problems and partial differential equations that map between manifolds via the closest point method}

￼\author[sfu]{Nathan D.~King\corref{ndkcor}}
\ead{nathank@sfu.ca}
\author[sfu]{Steven J.~Ruuth}
\ead{sruuth@sfu.ca}

 \address[sfu]{Department of Mathematics, Simon Fraser University, Burnaby, BC, Canada, V5A 1S6}
 
 \cortext[ndkcor]{Corresponding author}

\begin{abstract}
Maps from a source manifold $ {\mathcal M}$ to a target manifold ${\mathcal N}$ appear in liquid crystals, colour image enhancement, texture mapping, brain mapping, and many other areas. A numerical framework to solve variational problems and partial differential equations (PDEs) that map between manifolds is introduced within this paper. Our approach, the closest point method for manifold mapping, reduces the problem of solving a constrained PDE between manifolds ${\mathcal M}$ and ${\mathcal N}$ to the simpler problems of solving a PDE on ${\mathcal M}$ and projecting to the closest points on ${\mathcal N}.$ In our approach, an embedding PDE is formulated in the embedding space using closest point representations of ${\mathcal M}$ and ${\mathcal N}.$ This enables the use of standard Cartesian numerics for general manifolds that are open or closed, with or without orientation, and of any codimension.  An algorithm is presented for the important example of harmonic maps and generalized to a broader class of PDEs, which includes $p$-harmonic maps. Improved efficiency and robustness are observed in convergence studies relative to the level set embedding methods. Harmonic and $p$-harmonic maps are computed for a variety of numerical examples. In these examples, we denoise texture maps, diffuse random maps between general manifolds, and enhance colour images.     
\end{abstract}


\end{frontmatter}

\def\M{{\mathcal M}}
\def\N{{\mathcal N}}
\def\J{{\mathbf J}}
\def\H{{\mathbf H}}
\def\F{{\mathbf F}}
\def\G{{\mathbf G}}
\def\cp{\text{cp}}
\def\d{\text{d}}
\def\x{{\mathbf x}}
\def\y{{\mathbf y}}
\def\z{{\mathbf z}}
\def\u{{\mathbf u}}
\def\u{{\mathbf u}}
\newtheorem{prin}{Principle}
\newtheorem{defn}{Definition}
\newtheorem{theo}{Theorem}
\newtheorem{lem}{Lemma}

%\linenumbers

\section{Introduction}
\label{intro}
The need to compute maps from a source manifold $\M$ to a target manifold $\N$ is present in numerous fields, e.g., mathematical physics, image processing, computer vision, and medical imaging. In mathematical physics these types of maps occur in the study of liquid crystals~\cite{Virga1994}, micromagnetic materials~\cite{Landau1935}, biomembranes~\cite{Willmore1993}, and superconductors~\cite{Berlyand1999}. Applications in image processing and computer vision include colour image enhancement~\cite{Tang2001}, directional diffusion~\cite{Tang2000, Vese2002}, and texture mapping~\cite{Dinh2005}. The field of medical imaging contains applications such as brain image regularization~\cite{Memoli2004brain}, optic nerve head mapping~\cite{Gibson2010}, and brain mapping~\cite{Shi2007,Shi2009}.  

This paper introduces a numerical framework for solving variational problems and PDEs that define maps from a source manifold $\M$ to a target manifold $\N.$  Our primary concern is the development of numerical methods for PDEs derived from variational problems, i.e., the Euler-Lagrange equations. However, our approach also applies to more general PDEs. 

Intuition for numerical methods for manifold mapping problems can be gained from methods for {\it unconstrained PDEs on manifolds.} A PDE defined on a single manifold $\M$ is the special case when the solution $\u$ is not constrained to a target manifold $\N.$ One class of methods for such problems uses a smooth coordinate system or parameterization of the manifold.  In general, however, a substantial complication of the surface PDE can arise and artificial singularities can be introduced by the coordinate system~\cite{Floater2005}. A second approach solves the PDE on a triangulated representation of the manifold. There are numerous difficulties that can arise when using triangulations~\cite{Bertalmio2003}. In particular, there is no standard method for computing geometric primitives, e.g., tangents, normals, principal directions, and curvatures. The convergence of numerical methods on triangulated manifolds is also less understood compared to methods on Cartesian grids~\cite{Greer2006}.  

Another class of methods is  the {\it embedding methods}, which embed the surface PDE and solve in a narrow band surrounding the manifold.  The {\it embedding PDE} is constructed such that its solution, when restricted to the manifold, is the solution to the original surface PDE. An embedding method allows the use of standard Cartesian numerical methods when solving PDEs on complex surface geometries. Two main types of embedding methods have been developed: the level set method and the closest point method. Since these methods were developed for unconstrained PDEs on a manifold $\M,$ we denote them by $\text{LSM}_{\M}$ and $\text{CPM}_{\M},$ respectively. The $\text{LSM}_{\M}$  was introduced by Bertalm\'{i}o, Cheng, Osher and Sapiro~\cite{Bertalmio2001}. It represents the manifold as the zero level set of a higher dimensional function.  The $\text{CPM}_{\M}$ was introduced by Ruuth and Merriman~\cite{Ruuth2008}. It uses a closest point representation of the manifold.

 An obvious limitation of the $\text{LSM}_{\M}$ is that open manifolds with boundaries, or objects of codimension-two or higher, do not have a direct level set representation. Another difficulty arises when computations are localized to a band around the manifold. The introduction of boundaries at the edge of the computational domain leads to the use of artificial boundary conditions, which can degrade the convergence order; see~\cite{Greer2006} for the case of diffusion problems. On the other hand, the boundary values for the $\text{CPM}_{\M}$ are obtained from the manifold. This enables the use of banded computations without degrading the order of the underlying discretization.
 
Less work has been done on numerical methods for PDEs that map from $\M$ to $\N$ than on numerical methods for unconstrained PDEs on a single manifold. Notably, most of the numerical methods that have been developed compute harmonic maps for specific $\M$ and/or $\N.$ Numerical schemes of this type were first developed for the special case of $\N = S^{n-1},$ the unit hypersphere. See, for example,~\cite{Lin1989, Cohen1987,Dean1988,Cohen1989} for a number of algorithms that find stable solutions of harmonic maps onto $\N = S^{n-1}.$ One of the first algorithms proven to converge in a continuous setting was introduced by Alouges in~\cite{Alouges1997}. The algorithm was later proven to converge in a finite element setting, with acute triangles, by Bartels~\cite{Bartels2005}. Finite element methods  for more difficult problems have been developed, e.g., for $p$-harmonic maps~\cite{Barrett2007} and the Landau-Lifschitz-Gilbert equation~\cite{Alouges2014}. A finite element method for more general target manifolds has also been introduced; see~\cite{Bartels2010}. 

A different, parametric approach was taken by Vese and Osher~\cite{Vese2002} for $p$-harmonic maps onto $\N=S^{n-1}.$  Their method successfully denoises colour images, however, it is restricted to $\N=S^{n-1}.$  

 The $\text{LSM}_{\M}$  was extended by M\'{e}moli, Sapiro and Osher~\cite{Memoli2004} to solve variational problems and PDEs that define maps from $\M$ to $\N.$ This method will be denoted by $\text{LSM}_{\M}^{\N}$ throughout our paper. In a similar fashion, we extend the $\text{CPM}_{\M}$ to solve manifold mapping problems. Fundamental to our approach is the adoption of closest point representations of the source and target manifolds, $\M$ and $\N.$  This leads to improved geometric flexibility, as well as a means to avoid the introduction of artificial boundary conditions in banded computations. Since the method will handle problems that define maps between manifolds, it will be referred to as the closest point method for manifold mapping and will be denoted by $\text{CPM}_{\M}^{\N}.$

The paper is organized as follows.  We begin with a brief review of the original $\text{CPM}_{\M}$ for unconstrained PDEs on manifolds (Section~\ref{sec:cpm_review}). Section~\ref{sec:cpmmm} introduces our numerical framework for variational problems and PDEs that define maps from $\M$ to $\N,$ i.e., the $\text{CPM}_{\M}^{\N}$.  A comparison of the $\text{LSM}_{\M}^{\N}$ and the $\text{CPM}_{\M}^{\N}$ is given in Section~\ref{sec:compare}. The behaviour and performance of our method is illustrated with numerical examples in Section~\ref{sec:num_res}. In that section, noisy texture maps onto different target manifolds $\N$ are denoised. In addition, diffusion of a random map between general manifolds is shown and a method for colour image enhancement is illustrated. Section~\ref{sec:conc} gives conclusions and a discussion of possible future work.

%%%%%%%%%%%%%%%%%%%%%%%%%%%%%%%%%%%%%%%%%%%%%%
%%%%%%%%%%%%%%%%%%%%%%%%%%%%%%%%%%%%%%%%%%%%%%

\section{The closest point method for unconstrained PDEs on manifolds}
\label{sec:cpm_review}
 Our new algorithm is built on the explicit $\text{CPM}_{\M}$,~\cite{Ruuth2008}. We begin with a review of this method. An alternative, based on implicit time-stepping, is also possible; see~\cite{Macdonald2009} for further details on this method and its implementation.

 As the name suggests, the closest point method relies on a closest point representation of the manifold $\M.$ Closest point representations are less restrictive than  level set representations. A standard level set representation needs a well-defined inside and outside, which makes handling open manifolds and manifolds of codimension-two or higher more difficult. A closest point representation of the manifold $\M,$ in the embedding space $\mathbb{R}^m,$ assumes that for every $\x\in \mathbb{R}^m$ there exists a point $\cp_{\M}(\x)\in \M.$ The point $\cp_{\M}(\x)$ is the closest point on $\M$ to $\x$ in Euclidean distance:
\begin{defn}
Let $\x$ be some point in the embedding space $\mathbb{R}^m.$ Then, 
\begin{equation*}
\textnormal{\cp}_{\M}(\x) = \textnormal{arg}\hspace{0.01cm} \min_{\mathbf{z}\in\M} \|\x-\mathbf{z}\|_2
\end{equation*}
is the closest point of $\x$ to the manifold $\M.$
\label{def:cp}
\end{defn}

In general, the point $\cp_{\M}(\x)$ may not be unique. However, for a smooth manifold $\M$ it is unique if $\x$ is sufficiently close to $\M$~\cite{Macdonald2009,Marz2012}. Near such a smooth manifold, the closest point function and the well-known signed distance function $\d_{\M}$ of $\M$~\cite{Osher2006} are related via
\begin{equation}
\textnormal{\cp}_{\M}(\x) = \x - \textnormal{\d}_{\M}(\x) \nabla \textnormal{\d}_{\M}(\x).
\label{cp_func}
\end{equation}
The neighbourhood over which $\cp_{\M}$ is unique depends on the geometry of $\M,$ e.g., the size of its principal curvatures. Properties of the closest point function and calculus involving $\cp_{\M}$ have been investigated further by M\"{a}rz and Macdonald~\cite{Marz2012}. There they discuss the relationship between finitely smooth manifolds, finitely smooth functions on manifolds and PDE order. The definition of the closest point function is also extended to involve non-Euclidean distance.

 The $\text{CPM}_{\M}$ is an embedding method: it extends the problem defined on a manifold $\M$ to the embedding space $\mathbb{R}^m$ surrounding $\M.$ The $\text{CPM}_{\M}$ relies on two principles and the extension of surface data $u$ to construct an embedding PDE defined on $\mathbb{R}^m.$ Briefly, the intrinsic surface gradient $\nabla_{\M}$ and surface divergence $(\nabla_{\M} \hspace{0.02cm}\cdot)$ operators are replaced by the standard Cartesian gradient $\nabla$ and divergence $(\nabla \cdot)$ operators via the following principles~\cite{Ruuth2008}: 
\begin{prin}
Let $v$ be any function on $\mathbb{R}^m$ that is constant along normal directions of $\M.$ Then, at the surface, intrinsic gradients are equivalent to standard gradients, $\nabla_{\M} v = \nabla v.$
\label{prin1}
\end{prin}
\begin{prin}
Let $\mathbf{v}$ be any vector field on $\mathbb{R}^m$ that is tangent to $\M$ and tangent to all surfaces displaced by a fixed distance from $\M.$ Then, at the surface, $\nabla_{\M} \cdot \mathbf{v}  = \nabla \cdot \mathbf{v}.$
\label{prin2}
\end{prin}

Higher order derivatives can be handled by combining Principles~\ref{prin1} and \ref{prin2} with constant normal extensions of the surface data into the embedding space.  Constant normal extensions of the data are referred to as {\it closest point extensions} since they are implemented efficiently by composing surface data with the closest point function. That is, $u(\cp_{\M}(\x))$ is the closest point extension of $u$ at the point $\x \in \mathbb{R}^m.$ To illustrate this idea, consider the Laplace-Beltrami operator $\Delta_{\M} u = \nabla_{\M} \cdot (\nabla_{\M} u).$  If $u$ is a function defined on $\M,$ then $u(\cp_{\M})$ is constant along normal directions of $\M$ and therefore $\nabla_{\M} u = \nabla u(\cp_{\M})$ on $\M,$ by Principle~\ref{prin1}.  Principle~\ref{prin2} implies that $\nabla_{\M} \cdot(\nabla_{\M} u)=\nabla \cdot (\nabla u(\cp_{\M}))$ on $\M,$ since $\nabla_{\M} u$ is always tangent to the level sets of the distance function of $\M.$  In this fashion, an embedding PDE is obtained that involves standard Cartesian derivatives and a closest point function. 

The following steps detail the explicit $\text{CPM}_{\M}$ to solve PDEs on manifolds. First, a narrow banded computational domain, $\Omega_c,$ surrounding $\M$ is chosen and the initial surface data $u^0$ is extended onto $\Omega_c$ using the closest point extension. The following two steps are then alternated to obtain the explicit $\text{CPM}_{\M}$: 
\begin{itemize}
\item {\it Evolution.} The embedding PDE is solved on $\Omega_c$ for one time step (or one stage of a Runge-Kutta method).
\item {\it Closest point extension.} The solution on $\M$ is extended to the computational domain by replacing $u$ with $u(\cp_{\M})$ for all $\x\in\Omega_c.$
\end{itemize}

Note that the closest point extension defined in the second step involves interpolation. Interpolation is needed since $\cp_{\M}(\x)$ is not necessarily a grid point in $\Omega_c.$ The interpolation order depends on the derivative order $r$ and the differencing scheme order $q$ and should be chosen large enough to not produce errors greater than the differencing scheme. Following~\cite{Ruuth2008}, barycentric Lagrange interpolation is applied in a dimension-by-dimension fashion with polynomial degree $p=q+r-1$ in all our numerical examples.

For efficiency, computations should be localized to a banded region $\Omega_c$ surrounding the manifold. In our algorithms, a uniform hypercube grid is constructed around $\M$ and an indexing array is used to access points within a Euclidean distance $\lambda_c$ from $\M.$ The width of the computational band, $\lambda_c,$ depends on the degree of the interpolating polynomial $p,$  the differencing stencil, and the dimension of the embedding space $m.$  It is shown in~\cite{Ruuth2008} that for a second-order centred difference discretization of the Laplacian operator, 
\begin{equation*}
\lambda_c = \sqrt{(m-1)\left(\frac{p+1}{2}\right)^2 + \left(1+\frac{p+1}{2}\right)^2}\;\;\Delta x.
\end{equation*}

\section{Manifold mapping variational problems and PDEs}
\label{sec:cpmmm}
In this section, we introduce our framework for solving variational problems and PDEs that define maps from a source manifold $\M$ to a target manifold $\N.$ For clarity, we introduce the method for the case of harmonic maps. Other maps may also be approximated using our approach. We conclude this section by detailing an algorithm for these more general maps, which include $p$-harmonic maps. 

\subsection{Harmonic maps}
Harmonic maps~\cite{Moser2005,JostBook2011,Schoen1997} are important in many applications such as texture mapping~\cite{Dinh2005}, regularization of brain images~\cite{Memoli2004brain}, and colour image enhancement~\cite{Tang2001}.  Considerable research on the theory of harmonic maps has also been carried out, starting with the work of Fuller~\cite{Fuller1954} in 1954 and the more general theory by Eells and Sampson~\cite{Eells1964} in 1964. An important property of harmonic maps is their smoothness. They are also one of the most simple manifold mapping problems, whose study can provide insight into other mapping problems. Physically, a map is harmonic when $\M$ corresponds to a membrane that is constrained to $\N$ in elastic equilibrium~\cite{Eells1978}. 

We now give the mathematical definition of a harmonic map between two Riemannian manifolds $\M$ and $\N$~\cite{Memoli2004brain,Memoli2004}. Denote the signed distance functions of $\M$ and $\N$ as $\d_{\M}$ and $\d_{\N}$, respectively~\cite{Osher2006}. The intrinsic Jacobian of $\u$ is denoted by $\J_{\u}^{\d_{\M}}$ and can be written in terms of the standard Jacobian as $\J_{\u}^{\d_{\M}} = \J_{\u} \Pi_{\nabla \d_{\M}},$ where $\Pi_{\nabla \d_{\M}} = \mathbf{I} - \nabla \d_{\M} \nabla \d_{\M}^T$ is the projection operator onto the tangent space of $\M.$
\begin{defn}
Harmonic maps $\u:\M\rightarrow\N$ are the critical points of the Dirichlet energy  
\begin{equation}
E[\u] = \frac{1}{2}\int_{\M} \left\| \J_{\u}^{\d_{\M}}\right\|^2_{\mathcal{F}}  \;dv_{\M},
\label{imp_gen_energy}
\end{equation}
where $\| \cdot \|_{\mathcal{F}}$ is the Frobenius norm and $dv_\M$ is the volume element of $\M.$
\label{def:harm_map}  
\end{defn}

The map $\u:\M \rightarrow \N$ must be a $C^1$ map to ensure that $E[\u]$ is well-defined. Furthermore, by the Nash embedding theorem~\cite{Nash1954,Nash1956}, any Riemannian manifold $\M$ can be isometrically embedded in a higher dimensional Euclidean space $\mathbb{R}^m.$ Therefore, local coordinates on $\M$ and $\N$ can be written in terms of coordinates in Euclidean spaces $\mathbb{R}^m$ and $\mathbb{R}^n,$ respectively. That is, one can write $\u = (u_1,u_2, \ldots, u_n)^T$ with point-wise constraint $\u(\x)\in \N$ for any $\x= (x_1,x_2, \ldots, x_m)^T \in \M.$ 

M\'emoli et al.~\cite{Memoli2004} derived the Euler-Lagrange equations for~(\ref{imp_gen_energy}) in terms of the level set representation of $\N$ under the assumption that $\M$ is flat and open. The same calculation is carried out by Moser~\cite{Moser2005} in terms of the closest point representation of $\N.$ There, the closest point function is called the nearest point projection and is used to prove regularity results of harmonic maps (see Chapter 3 of~\cite{Moser2005}). The Euler-Lagrange equations corresponding to~(\ref{imp_gen_energy}), assuming $\M$ is flat and open, are
\begin{equation}
\Delta \u - \sum_{\ell = 1}^m \H_{\cp_{\N}(\u)} \left[\frac{\partial \u}{\partial x_{\ell}}, \frac{\partial \u}{\partial x_{\ell}} \right] = 0,
\label{cp_EL}
\end{equation}
where the notation $\mathbf{A} [\y, \z] = (\z^T \mathbf{A}^1 \y,\z^T \mathbf{A}^2 \y, \ldots,\z^T \mathbf{A}^n \y)^T$ is used. The matrix $\H_{\cp_{\N}(\u)}$ denotes the Hessian of $\cp_{\N}(\u),$ i.e., the Hessian of each component of $\cp_{\N}(\u)$ is  $\H^i_{\cp_{\N}(\u)}$ for $i=1,2,\ldots,n.$  To illustrate the process, \ref{app:liquid_cry} derives the Euler-Lagrange equations~(\ref{cp_EL}) for the important case where $\M$ is a flat, open subset of $\mathbb{R}^m$ and $\N=S^{n-1}.$ This corresponds to the application of liquid crystals~\cite{Virga1994}.

A solution to~(\ref{cp_EL}) could be obtained by evolving the corresponding gradient descent flow to steady state (cf.~\cite{Memoli2004, Moser2005}). The gradient descent flow is a PDE that introduces an artificial time variable and evolves in the direction of maximal decrease of the energy. Numerically, one could discretize this gradient descent flow and evolve the solution until some long time $t_f.$ A simpler approach to numerically approximate the harmonic map via a gradient descent flow is given next. We shall see that our approach has the further benefit of handling more general variational problems and PDEs, including $p$-harmonic maps.

\subsection{The closest point method for manifold mapping}
\label{subsec:cpmmm}
To design a numerical method  we do not discretize~(\ref{cp_EL}). Instead, we write the Euler-Lagrange equations for~(\ref{imp_gen_energy}) as $\Pi_{T_{\u} \N} (\Delta_{\M} \u) = 0$~\cite{Schoen1997}, where $\Pi_{T_{\u} \N}$ is the projection operator at the point $\u$ onto the tangent space of $\N.$  The vector $\Delta_{\M} \u$ is defined component-wise, i.e., $\Delta_{\M} \u = (\Delta_{\M} u_1, \Delta_{\M} u_2, \ldots, \Delta_{\M} u_n)^T.$ The corresponding gradient descent flow is
\begin{equation}
\left\{ \begin{aligned}
&\frac{\partial \u}{\partial t} = \Pi_{T_{\u} \N} (\Delta_{\M} \u),\\
&\u (\x,0) = \u^0 (\x),\\
&\J_{\u}^{\d_{\M}} \mathbf{n}|_{\partial \M} = 0,
\end{aligned} \right.
\label{proj_graddes}
\end{equation}
where $\u^0(\x)$ is a given initial map. A justification for the homogeneous Neumann boundary conditions is given in Appendix A of~\cite{Memoli2004}.

To discretize~(\ref{proj_graddes}), intrinsic geometric quantities are replaced by terms involving standard Cartesian coordinates and closest point functions. As we saw in Section~\ref{sec:cpm_review}, the term $\Delta_{\M} \u$ can be replaced by $\Delta \u(\cp_{\M})$ using the $\text{CPM}_{\M}.$ Furthermore, the projection operator $\Pi_{T_{\u}\N}$ equals the Jacobian of the closest point function, $\J_{\cp_{\N}(\u)},$ for $\u\in \N$~\cite{Marz2012,Moser2005}. Applying these replacements gives the embedding gradient descent flow
\begin{equation}
\frac{\partial \u}{\partial t} =\J_{\cp_{\N}(\u)} (\Delta \u(\cp_{\M})).
\label{emb_proj_graddes}
\end{equation}
 
New identities may be required to formulate an embedding PDE for more general variational problems and PDEs.  However, the general procedure is the same in all cases: rewrite geometric quantities intrinsic to $\M$ and $\N$ in terms of $\cp_{\M}$ and $\cp_{\N},$ respectively. 
 
The closest point function, $\cp_{\N},$ is itself a projection operator onto $\N.$ By splitting the evolution of~(\ref{emb_proj_graddes}) into two steps we can eliminate the computation of $\J_{\cp_{\N}}$ and further simplify the numerical method. More generally, a splitting can be formulated for any PDE with intrinsic geometric terms on $\M$ that are projected onto the tangent space of $\N,$ e.g., PDEs of the form
\begin{equation}
\frac{\partial \u}{\partial t} = \Pi_{T_{\u}\N}  (\mathbf{F}(\x,\u, \nabla_{\M} \u, \nabla_{\M} \cdot( \nabla_{\M} \u), \ldots)).
\label{gen_proj}
\end{equation} 

 To solve~(\ref{gen_proj}), we first evolve an embedding PDE on $\M,$ 
 \begin{equation*}
 \frac{\partial \tilde{\u}}{\partial t} =  \mathbf{F}(\cp_{\M},\tilde{\u}(\cp_{\M}), \nabla\tilde{\u}(\cp_{\M}), \nabla \cdot( \nabla \tilde{\u}(\cp_{\M})), \ldots),
%\label{emb_gen_proj}
\end{equation*}  
 for one time step of size $\Delta t$ to give $\u_{\text{ext}}(\x)$ at each grid node $\x \in \Omega_c.$ We emphasize that this step omits the projection $\Pi_{T_{\u} \N}$ (equivalently $\J_{\cp_{\N}(\u)}$) appearing in~(\ref{gen_proj}). The second step projects $\u_{\text{ext}}(\x)$ onto $\N$ via $\cp_{\N}(\u_{\text{ext}}(\x)).$ The result, $\u^{k+1}(\x),$ approximates the solution $\u(\x,t^{k+1}),$ $t^{k+1} = (k+1)\Delta t,$ to~(\ref{gen_proj}) at points $\x \in \M.$ Starting from $\u^0(\x) = \u(\cp_{\M}(\x),0),$ the steps of the $\text{CPM}_{\M}^{\N}$ to advance from time $t^k$ to time $t^{k+1}$ are given explicitly by Algorithm~\ref{alg:cpmmm} below.
\begin{algorithm}
\caption{A time step of the $\text{CPM}_{\M}^{\N}$ for $ \partial \u/\partial t= \Pi_{T_{\u}\N} (\mathbf{F}(\x,\u,\nabla_{\M} \u,\nabla_{\M} \cdot(\nabla_{\M} \u), \ldots)),$ starting from $\u^k(\x).$}
\label{alg:cpmmm}
\begin{algorithmic}
\STATE 1. Solve  $\partial \tilde{\u}/\partial t =  \mathbf{F}(\x,\tilde{\u},\nabla_{\M} \tilde{\u},\nabla_{\M} \cdot(\nabla_{\M} \tilde{\u}), \ldots)$ for one time step of size $\Delta t$ using the $\text{CPM}_{\M}:$
\begin{itemize}
\item {\it Evolution.} For $\x \in \Omega_c$ solve
\begin{equation*} 
\left\{ \begin{aligned}
&\frac{\partial \tilde{\u}}{\partial t}(\x,t) = \mathbf{F}(\x,\tilde{\u}(\x,t), \nabla \tilde{\u}(\x,t), \nabla \cdot( \nabla \tilde{\u}(\x,t)), \ldots),\\
&\tilde{\u} (\x,0) = \u^k(\x),\\
\end{aligned}\right.
\end{equation*} 
for one time step.
\item {\it Closest point extension.} Set $\u_{\text{ext}}(\x) =  \tilde{\u}(\cp_{\M}(\x),\Delta t).$
\end{itemize}

\STATE 2. Project $\u_{\text{ext}}(\x)$ onto $\N$ by setting $\u^{k+1}(\x) = \cp_{\N}(\u_{\text{ext}}(\x)).$
\end{algorithmic}
\end{algorithm}

In this paper, we apply forward Euler time-stepping in Step 1 of the $\text{CPM}_{\M}^{\N},$ however, other explicit~\cite{Ruuth2008} or implicit~\cite{Macdonald2009} choices may be used. Note also that the homogeneous Neumann boundary conditions in~(\ref{proj_graddes}) are imposed automatically by the $\text{CPM}_{\M}$~\cite{Ruuth2008}. Therefore, Step 1 of the $\text{CPM}_{\M}^{\N}$ does not involve direct implementation of boundary conditions when $\M$ is an open manifold (i.e., a manifold with boundaries). 

In the harmonic mapping case~(\ref{proj_graddes}), $\mathbf{F} = \Delta_{\M} \u.$ Another important special case is the $p$-harmonic maps. The extremizing functions $\u:\M\rightarrow \N$ of the energy
\begin{equation}
E_p[\u] = \int_{\M} e_p[\u] dv_{\M},
\label{pharm_en}
\end{equation}
with $1\leq p <\infty$ and
\begin{equation*}
e_p[\u] = \frac{1}{p} \left\|\J_{\u}^{\d_{\M}}\right\|^p_{\mathcal{F}},
\end{equation*}
are called $p$-harmonic maps. The gradient descent flow for the energy~(\ref{pharm_en}) is~\cite{Memoli2004},
 \begin{equation}
\left\{ \begin{aligned}
&\frac{\partial \u}{\partial t} = p^{1-\frac{2}{p}}\;\Pi_{T_{\u} \N} \left(\nabla \cdot \left(\left(e_p[\u]\right)^{1-\frac{2}{p}} \J_{\u}^{\d_{\M}}\right)\right),\\
&\u (\x,0) = \u^0 (\x),\\
&\J_{\u}^{\d_{\M}} \mathbf{n}|_{\partial \M} = 0,
\end{aligned} \right.
\label{pharm_proj_graddes}
\end{equation}
where the divergence of the matrix is defined as the divergence of each row of the matrix. Noting that $\J_{\u}^{\d_{\M}} = (\nabla_{\M} \u)^T,$ we obtain the embedding form of~(\ref{pharm_proj_graddes})
 \begin{equation}
 \frac{\partial \u}{\partial t} = p^{1-\frac{2}{p}}\;\J_{\cp_{\N}(\u)} \left(\nabla \cdot \left( \left(\frac{1}{p} \left\|(\nabla \u(\cp_{\M}))^T\right\|^p_{\mathcal{F}}\right)^{1-\frac{2}{p}} (\nabla \u(\cp_{\M}))^T\right)\right),
\label{emb_pharm_proj_graddes}
\end{equation}
which can be evolved using the $\text{CPM}_{\M}^{\N}$ (Algorithm~\ref{alg:cpmmm}).
  
We conclude this subsection by showing the consistency of the $\text{CPM}_{\M}^{\N}$ applied to~(\ref{gen_proj}) in Theorem~\ref{theo1}. The proof of Theorem~\ref{theo1} uses the following lemma, which is a specific case of Taylor's theorem~\cite{Coleman2012} for normed vector spaces.
\begin{lem}
Let $A$ and $B$ be normed vector spaces and $A_O$ an open subset of $A.$ Suppose that  $\mathbf{a} \in A_O$ and $\mathbf{h}\in A$ such that the segment $[\mathbf{a}, \mathbf{a}+\mathbf{h}]\in A_O.$ Let $\mathbf{f}:A_O\rightarrow B$ be a $C^1$ mapping whose Hessian, $\mathbf{H}_{\mathbf{f}},$ is finite. Then,
\begin{equation*}
\mathbf{f}(\mathbf{a}+\mathbf{h}) = \mathbf{f}(\mathbf{a}) + \J_{\mathbf{f}(\mathbf{a})} \mathbf{h} + \mathcal{O}\left(\|\mathbf{h}\|^2\right).
\end{equation*}
\label{lem1}
\end{lem}
\begin{theo}
Let $\M\subset \mathbb{R}^m$ be a smooth manifold. Suppose in a neighbourhood of $\M\times[0,\Delta t]$ that $\tilde{\u}: \mathbb{R}^m \times \mathbb{R} \rightarrow \mathbb{R}^n$ is a $C^1$ mapping with a finite Hessian. Further assume that  $\textnormal{cp}_{\N}:\mathbb{R}^n \rightarrow \mathbb{R}^n$ is a $C^1$ mapping with a finite Hessian in a neighbourhood of $\N.$ Then, the $\textnormal{CPM}_{\M}^{\N}$ (Algorithm~\ref{alg:cpmmm}) is consistent with the PDE~(\ref{gen_proj}) for any $\x \in \M.$
\label{theo1}
\end{theo}
 \begin{proof}
Let $A = \mathbb{R}^m \times \mathbb{R},$ $B = \mathbb{R}^n,$ $\mathbf{a}=(\cp_{\M}(\x),0),$ $\mathbf{h} = (0,\Delta t),$ and $\x\in\mathbb{R}^m.$ Let $A_O$ be a neighbourhood of $\M \times [0,\Delta t]$ where $\tilde{\u}$ is $C^1$ and $\H_{\tilde{u}}$ is finite. Using Lemma~\ref{lem1}, we expand $\tilde{\u}(\cp_{\M}(\x),\Delta t)$ and substitute $\partial \tilde{\u}( \cp_{\M}(\x),0)/\partial t= \mathbf{F}\left( \cp_{\M}(\x),\tilde{\u}( \cp_{\M}(\x),0), \nabla \tilde{\u}(\cp_{\M}(\x),0), \ldots \right)$ to obtain
\begin{equation*}
\u_{\text{ext}}(\x) = \tilde{\bf u}(\cp_{\M}(\x),\Delta t) = \tilde{\u}(\cp_{\M}(\x),0) + \Delta t \F(\cp_{\M}(\x),\tilde{\u}(\cp_{\M}(\x),0), \nabla \tilde{\u}(\cp_{\M}(\x),0), \ldots) + \mathcal{O}\left(\Delta t^2\right). 
\end{equation*}
 The numerical approximation at time $t^{k+1}$ can therefore be expressed as
\begin{align*}
\u^{k+1}(\x) = \cp_{\N}\left(\tilde{\u}(\cp_{\M}(\x),0) + \Delta t \F(\cp_{\M}(\x),\tilde{\u}(\cp_{\M}(\x),0), \nabla \tilde{\u}(\cp_{\M}(\x),0), \ldots)+\mathcal{O}\left(\Delta t^2\right)\right). 
\end{align*}
Applying Lemma~\ref{lem1} to expand the closest point function, with $A = B =  \mathbb{R}^n,$ $\mathbf{a}=\tilde{\u}(\cp_{\M}(\x),0),$ and $\mathbf{h} = \Delta t \F\left(\cp_{\M}(\x)\right. ,$ \\  $\left. \tilde{\u}( \cp_{\M}(\x),0),\nabla \tilde{\u}(\cp_{\M}(\x),0), \ldots\right)+\mathcal{O}\left(\Delta t^2\right),$ yields 
\begin{equation*}
\begin{aligned}
\u^{k+1}(\x) &=  \cp_{\N}(\tilde{\u}(\cp_{\M}(\x),0))\\
& +\J_{\cp_{\N}(\tilde{\u}( \cp_{\M}(\x),0))} \left(\Delta t \F\left(\cp_{\M}(\x),\tilde{\u}( \cp_{\M}(\x),0),\nabla \tilde{\u}(\cp_{\M}(\x),0), \ldots\right)+\mathcal{O}\left(\Delta t^2\right)\right) + \mathcal{O}\left(\Delta t^2\right),\\
&=  \u^k(\cp_{\M}(\x))+\Delta t \J_{\cp_{\N}\left(\u^k( \cp_{\M}(\x))\right)} \left( \F\left( \cp_{\M}(\x),\u^k( \cp_{\M}(\x)), \nabla \u^k(\cp_{\M}(\x)), \ldots\right)\right) +\mathcal{O}\left(\Delta t^2\right),
\end{aligned} 
\end{equation*}
where we have used $\tilde{\u}(\x,0) = \u^k(\x)$ and $\cp_{\N}(\u^k) = \u^k.$ We apply $\J_{\cp_{\N}\left(\u^k\right)} = \Pi_{T_{\left(\u^k\right)}\N}$ and Principles~\ref{prin1} and~\ref{prin2} of the $\text{CPM}_{\M}$ to obtain
\begin{equation*}
\u^{k+1} =  \u^k +\Delta t \Pi_{T_{\left(\u^k\right)}\N} \left(\F\left( \x,\u^k, \nabla_{\M} \u^k, \nabla_{\M} \cdot (\nabla_{\M} \u^k), \ldots\right)\right) +\mathcal{O}\left(\Delta t^2\right),
\end{equation*}
 which holds for any $\x\in\M.$ Rearranging and taking the limit as $\Delta t \rightarrow 0$ obtains the desired result
\begin{equation*}
\frac{\partial \u}{\partial t} =  \Pi_{T_{\u}\N} (\mathbf{F}(\x,\u, \nabla_{\M} \u, \nabla_{\M} \cdot (\nabla_{\M} \u), \ldots)).
\label{equi}
\end{equation*} 
\end{proof}
  
In summary, we may evolve~(\ref{gen_proj}) by alternating between a step of  PDE evolution on $\M$ (via the $\text{CPM}_{\M}$) and an evaluation of the closest point function for $\N.$ Properties of this closest point method for manifold mapping, $\text{CPM}_{\M}^{\N}$, are considered in some detail next. Particular attention will be paid to the performance of the algorithm relative to its closest algorithmic companion, the level set method for manifold mapping, $\text{LSM}_{\M}^{\N}$.

%%%%%%%%%%%%%%%%%%%%%%%%%%%%%%%%%%%%%%%%%%%%%%%
%%%%%%%%%%%%%%%%%%%%%%%%%%%%%%%%%%%%%%%%%%%%%%%

\section{A comparison: the closest point and level set methods for manifold mapping}
\label{sec:compare}
In this section we compare the closest point and level set methods for manifold mapping, i.e., the $\text{CPM}_{\M}^{\N}$ and the $\text{LSM}_{\M}^{\N}$. The comparison is performed for the problem of computing harmonic maps $\u(\x): \M \rightarrow \N.$  Both methods compute the harmonic map by numerically approximating the gradient descent flow
\begin{equation}
\frac{\partial \u}{\partial t} = \Pi_{T_{\u} \N} ( \Delta_{\M} \u),
\label{harm_grad}
\end{equation} 
until steady state. 

\subsection{A discretization for harmonic maps}
To begin, we select a uniform computational grid $\Omega_{c}$ surrounding the manifold $\M$.  Assume $\M$ and $\N$ are manifolds embedded in $\mathbb{R}^m$ and $\mathbb{R}^n,$ respectively. Then $\x = (x_1, x_2, \ldots, x_m)^T \in \M$ and $\u(\x) = (u_1(\x), u_2(\x),\ldots, u_n(\x))^T \in \N.$ For the discretization of~(\ref{harm_grad}), denote discrete point locations by $\x_j$ and the approximate solution at time $t^k$ by $\u^k.$

We now compare and contrast the $\text{CPM}_{\M}^{\N}$  and the $\text{LSM}_{\M}^{\N}$. The derivation of the $\text{CPM}_{\M}^{\N}$ defines an embedding PDE~(\ref{emb_proj_graddes}) from~(\ref{harm_grad}) using 
\begin{enumerate}
\item   $\Pi_{T_{\u} \N} = \J_{\cp_{\N}(\u)}$ for any $\u \in \N,$
\item $\Delta_{\M} \u(\x) = \Delta \u(\cp_{\M}(\x))$ on $\M.$
\end{enumerate}
By splitting the evolution of the embedding PDE and replacing the projection by a closest point evaluation, we obtain the $\text{CPM}_{\M}^{\N}$. In contrast, the embedding PDE for the $\text{LSM}_{\M}^{\N}$ is derived using different properties based on the level set functions $\phi$ and $\psi$ representing $\M$ and $\N,$ respectively. Specifically, the $\text{LSM}_{\M}^{\N}$ uses
\begin{enumerate}
\item   $\Pi_{T_{\u} \N} = \Pi_{\nabla \psi(\u)}$ for any $\u \in \N,$
\item $\Delta_{\M} \u(\x) = \nabla \cdot (\Pi_{\nabla \phi(\x)} \nabla \u(\x))$ on $\M.$
\end{enumerate}

Denote forward and backward discretizations of the gradient by  $\nabla^+$ and $\nabla^-,$ respectively. Further denote the discretization of the Laplacian as $\Delta_0.$ Using forward Euler time-stepping, the $\text{CPM}_{\M}^{\N}$ (Algorithm~\ref{alg:cpmmm}) and the $\text{LSM}_{\M}^{\N}$ for~(\ref{harm_grad}) can be implemented as
\begin{itemize}
\item $\text{CPM}_{\M}^{\N}$: 
\begin{equation}
\u^{k+1}(\x_j) = \cp_{\N}\left(\u^k(\cp_{\M}(\x_j)) + \Delta t \left(\Delta_0 \u^k(\cp_{\M}(\x_j))\right)\right),
\label{num_cp}
\end{equation}
\item $\text{LSM}_{\M}^{\N}$: 
\begin{equation*}
\u^{k+1}(\x_j) = \u^k(\x_j) + \Delta t \mathbf{P}\left(\u^k(\x_j)\right) \left(\nabla^- \cdot \left(\mathbf{Q}(\x_j) \nabla^+ \u^k(\x_j) \right)  \right),
\end{equation*}
\end{itemize}
where $\mathbf{Q}$ and $\mathbf{P}$ are respectively discrete representations of $\Pi_{\nabla \phi}$ and $\Pi_{\nabla \psi}.$ Truncation errors in the $\text{LSM}_{\M}^{\N}$ can cause $\u^{k+1}(\x_j)$ to leave the target manifold $\N.$ M\'emoli et al.~\cite{Memoli2004} address this problem by projecting the solution back onto the target manifold $\N$ after each time step. The $\text{LSM}_{\M}^{\N}$ implemented in practice is therefore
\begin{equation}
\u^{k+1}(\x_j) = \cp_{\N}\left(\u^k(\x_j) + \Delta t \mathbf{P}\left(\u^k(\x_j)\right) \left(\nabla^- \cdot \left(\mathbf{Q}(\x_j) \nabla^+ \u^k(\x_j) \right)  \right)\right).
\label{ls_prac_imp}
\end{equation}

\subsection{Comparison of the methods}
Seven different areas of comparison between the $\text{CPM}_{\M}^{\N}$ and the $\text{LSM}_{\M}^{\N}$ are considered: manifold generality, implementation difficulty, convergence, computational work, memory requirements, accuracy, and convergence rate. The $\text{CPM}_{\M}^{\N}$ is better in all these aspects as will be discussed below. 

The $\text{CPM}_{\M}^{\N}$ can handle more complex surface geometries. A closest point function can be defined for manifolds that are open or closed, with or without orientation, and of any codimension. A level set function is most natural for closed manifolds of codimension-one. It can be extended to more general manifolds but this requires multiple level set functions, which complicates implementation and analysis (see, e.g., Section 3 of~\cite{Memoli2004}). 

The implementation of the $\text{CPM}_{\M}^{\N}$ is simpler since it does not form the projection matrix $\mathbf{P}.$ Also, for the harmonic mapping problem, one can discretize the Laplacian directly instead of discretizing gradients and forming $ \mathbf{Q}.$ A further reason the $\text{CPM}_{\M}^{\N}$ is easier to implement is that its closest point extension step involves only standard interpolation. The $\text{LSM}_{\M}^{\N}$ performs an ``extension evolution'' to extend surface data such that $\nabla \u \cdot \nabla \phi = 0.$ Typically, this step is carried out via a fast marching method or by evolving the gradient descent flow 
\begin{equation}
\frac{\partial \u}{\partial t} + \text{sign}(\phi)(\nabla \u \cdot \nabla \phi) = 0,
\label{reext}
\end{equation}
 to steady state. 
Note that, for some specific PDEs, the extension evolution for the $\text{LSM}_{\M}$ is only required once to prepare the initial data~\cite{Bertalmio2001}.  See also \cite{Greer2006}, where a modified projection matrix $\mathbf{Q}$ is introduced into
the $\text{LSM}_{\M}$ to yield a method that does not require any data re-extension. 
%The modified level set method introduces more complexities by modifying the projection matrix $\mathbf{Q}$ using the Hessian of $\phi$~\cite{Greer2006}. 
In contrast, we emphasize that the $\text{CPM}_{\M}$ requires a closest point extension at every time step.
 
For stationary surfaces, the extension evolution step for the $\text{LSM}_{\M}^{\N}$ can be computed efficiently using the closest point extension.
% the closest point extension of $\u$ satisfies $\nabla \u \cdot \nabla \phi = 0.$ 
Computational efficiency is prioritized over memory use in our numerical examples. Therefore, we use the closest point extension in the $\text{LSM}_{\M}^{\N}.$  In our implementation, the closest point extension is a small part of the overall computational cost of the $\text{LSM}_{\M}^{\N}$.  For example, it accounts for less than 1.4\% of the cost in the $\text{LSM}_{\M}^{\N}$ computations described in Section~\ref{subsec:ident}.

The $\text{CPM}_{\M}^{\N}$ performs better with respect to convergence. Theoretically, when applied to the harmonic mapping problem, the $\text{LSM}_{\M}^{\N}$ leads to degenerate PDEs~\cite{Greer2006}. This degeneracy can have an adverse effect on discretizations and little is known about the convergence of such schemes~\cite{Greer2006}. In contrast, the $\text{CPM}_{\M}^{\N}$ involves standard heat flow in its evolution step. This is discretized using standard Cartesian numerical methods in the embedding space. We further note that the $\text{CPM}_{\M}^{\N}$ has superior stability characteristics in practice.

Computational work per step is another natural area for comparison between the $\text{CPM}_{\M}^{\N}$ and the $\text{LSM}_{\M}^{\N}$. Comparing~(\ref{num_cp}) with the implemented version of the $\text{LSM}_{\M}^{\N}$, (\ref{ls_prac_imp}), we need only consider expressions within $\cp_{\N}(\cdot ).$ For the $\text{CPM}_{\M}^{\N}$~(\ref{num_cp}) we have a discretization on the manifold $\M$ alone:  $$\u^k(\cp_{\M}(\x_j)) + \Delta t \left(\Delta_0 \u^k(\cp_{\M}(\x_j))\right).$$ The $\text{LSM}_{\M}^{\N},$ however, is a discretization of the original PDE~(\ref{harm_grad}) between two manifolds: $$\u^k(\x_j) + \Delta t \mathbf{P}\left(\u^k(\x_j)\right) \left(\nabla^{-} \cdot \left(\mathbf{Q}(\x_j) \nabla^{+} \u^k(\x_j) \right)  \right),$$  which involves the obvious added work of constructing and applying the projection matrix $\mathbf{P}\left(\u^k(\x_j)\right).$ Moreover, two discretization matrices are applied, $\nabla^{-}$ and $\nabla^{+},$ instead of the one $\Delta_0.$ The $\text{LSM}_{\M}^{\N}$ also requires, in general, an extension evolution to be performed every $\ell$ time steps ($\ell=5$ in~\cite{Memoli2004}) to obtain stable results. This extension evolution step not only adds work, but is a further source of error. Finally, we note that the $\text{CPM}_{\M}^{\N}$ uses a relatively small, analytically defined computational band. This yields further computational savings over the $\text{LSM}_{\M}^{\N}$; see Section~\ref{subsec:ident} for further details.

The memory requirements per time step of the $\text{CPM}_{\M}^{\N}$ are less than that of the $\text{LSM}_{\M}^{\N}.$ 
%The extra projection and discretization matrices required for the $\text{LSM}_{\M}^{\N}$ add more memory overhead. 
Both the $\text{CPM}_{\M}^{\N}$ and $\text{LSM}_{\M}^{\N}$ store the solution vector $\u^k,$ the closest point function $\cp_{\N},$ and a matrix, $\mathbf{E},$ that applies the closest point extension for $\M.$ 
%(Note that the closest point function $\cp_{\M}$ is not stored since it is only needed to construct $\mathbf{E}.$)
The $\text{CPM}_{\M}^{\N}$ stores one tridiagonal discretization matrix, while  the $\text{LSM}_{\M}^{\N}$ stores two bidiagonal discretization matrices. %(All the matrices $\Delta_0,$ $\nabla^{-},$ $\nabla^{+},$ and $\mathbf{E}$ should take advantage of sparsity to reduce memory usage and improve efficiency.)  
In addition, the $\text{LSM}_{\M}^{\N}$ requires the storage of $\mathbf{Q},$ $\mathbf{P}(\u^k),$ as well as the matrices used to form $\mathbf{P}(\u^k)$ at each time step. 
This last group of matrices depends on how $\mathbf{P}(\u^k)$ is formed, e.g., via interpolation of $\psi$ (which is generally more efficient) or interpolation of $\mathbf{P}(\u^0).$ Finally, we observe that all the matrices for the $\text{LSM}_{\M}^{\N}$ are larger than the $\text{CPM}_{\M}^{\N}$ matrices because the $\text{LSM}_{\M}^{\N}$ requires a larger computational band to obtain the expected convergence rate; see Section~\ref{subsec:ident}. 

Accuracy and convergence rate are our final areas for comparison. In practical computations embedding methods must localize computations to a band around the manifold. In the $\text{LSM}_{\M}^{\N}$ this leads to the imposition of artificial boundary conditions that degrade the accuracy and convergence rate~\cite{Greer2006}. The $\text{CPM}_{\M}^{\N}$ obtains values at the boundary of the band from the manifold as part of the closest point extension step.  

Accuracy, convergence and computational work in practice are also of interest; we compare these next in numerical experiments for the problem of computing identity maps.

\subsubsection{Identity maps}
\label{subsec:ident}
The identity map $\u(\x): \M \rightarrow \N,$ with $\M=\N,$ is a harmonic map~\cite{JostBook2011}. Identity maps are a natural choice for conducting convergence studies since the exact harmonic map, $\u(\x) = \x,$ is known. We now provide convergence studies of identity map computations for the unit sphere, an ellipsoid, and a torus. 

In each case, an initial noisy map $\u^0(\x)$ is evolved to steady state.  To construct $\u^0(\x),$ a normally distributed random map ${\bf r}(\x)$ in $\mathbb{R}^3$ is added to points on $\N.$ The points are then projected back onto $\N$ using  $\cp_{\N},$ i.e.,  
\begin{equation*}
\u^0(\x) = \cp_{\N} ({\u}(\x) + {\bf r}(\x)).
\end{equation*} 
Each component of the random map ${\bf r}(\x)$ is set equal to $\alpha \cdot \texttt{randn},$ where \texttt{randn} is a M\textsc{atlab} command that returns a random scalar drawn from the standard normal distribution. The scalar $\alpha$  is a constant that controls the size of the noise; the value $\alpha=0.05$ is selected in the following convergence studies. 

The addition of random noise produces different convergence rates in each experiment. Tables~\ref{tab:ls_cp_sphere_sphere}-\ref{tab:ls_cp_torus_torus} therefore show convergence rates averaged over 96 realizations of the computation. Second-order centred finite differences were used for $\Delta_0$ and first-order differences were used for $\nabla^+$ and $\nabla^-.$ The time step-size was $\Delta t = 0.1 \Delta x^2,$ where $\Delta x$ is the spatial step-size. The maximum Euclidean distance between $\u^k$ and $\u$ over nodes $\mathbf{z}_j \in \M$ was used as the error,  i.e., $$\text{Error} = \hspace{0.01cm} \max_{\mathbf{z}_j}   \|\u^k(\mathbf{z}_j) - \mathbf{z}_j\|_2.$$  The errors from the $\text{CPM}_{\M}^{\N}$~(\ref{num_cp}) and the $\text{LSM}_{\M}^{\N}$~(\ref{ls_prac_imp}) were computed at the final time $t_f = 0.01.$ See Tables~\ref{tab:ls_cp_sphere_sphere}-\ref{tab:ls_cp_torus_torus} for the results.
\begin{table}
\centering
\begin{tabular}{p{1.25cm} K{1.75cm} K{1.75cm} K{1.5cm} K{1.5cm} K{1.75cm} K{1.75cm} c}
\hline
$\Delta x$ &\multicolumn{2}{c}{Error} & \multicolumn{2}{c}{Conv. rate} & \multicolumn{2}{c}{Comp. time} & Speedup\\
\hline
& $\text{LSM}_{\M}^{\N}$ & $\text{CPM}_{\M}^{\N}$ & $\text{LSM}_{\M}^{\N}$ &$\text{CPM}_{\M}^{\N}$  & $\text{LSM}_{\M}^{\N}$ & $\text{CPM}_{\M}^{\N}$& \\
\hline
$0.2$ & 7.54e$-$02  & 5.28e$-$02 &  & & 0.997 s & 0.049 s &20.5\\
$0.1$ & 3.60e$-$02 &  2.66e$-$02 & $1.07$& $0.99$ & 5.214 s & 0.539 s &9.7\\
$0.05$ & 1.65e$-$02 &  1.39e$-$02 &$1.12$ & $0.94$ & 69.15 s  & 8.875 s & 7.8\\
$0.025$ & 7.93e$-$03 &  6.74e$-$03 & $1.05$& $1.05$ & 1091 s & 143.4 s &7.6\\
$0.0125$ & 4.06e$-$03 & 3.50e$-$03 &$0.97$ & $0.94$ & 18326 s &  2458 s & 7.5\\
\hline
\end{tabular}
\caption{Convergence study for a unit sphere identity map using the $\text{LSM}_{\M}^{\N}$  and the $\text{CPM}_{\M}^{\N}$ with $t_f = 0.01.$ }
\label{tab:ls_cp_sphere_sphere}
\end{table}

\begin{table}
\centering
\begin{tabular}{p{1.25cm} K{1.75cm} K{1.75cm} K{1.5cm} K{1.5cm} K{1.75cm} K{1.75cm} c}
\hline
$\Delta x$ &\multicolumn{2}{c}{Error} & \multicolumn{2}{c}{Conv. rate} & \multicolumn{2}{c}{Comp. time} &Speedup\\
\hline
& $\text{LSM}_{\M}^{\N}$ & $\text{CPM}_{\M}^{\N}$ & $\text{LSM}_{\M}^{\N}$ &$\text{CPM}_{\M}^{\N}$  & $\text{LSM}_{\M}^{\N}$ & $\text{CPM}_{\M}^{\N}$& \\
\hline
$0.2$ & 7.30e$-$02& 5.26e$-$02 & & &  2.857 s &0.216 s & 13.2 \\
$0.1$ & 3.37e$-$02& 2.64e$-$02 & $1.11$ & $0.99$ & 17.45 s &2.231 s  & 7.8\\
$0.05$ & 1.61e$-$02& 1.37e$-$02 & $1.06$ &  $0.95$ & 191.7 s & 33.85 s &5.7\\
$0.025$ & 7.91e$-$03 & 6.80e$-$03 & $1.02$ & $1.01$ & 2730 s & 534.9 s &5.1\\
$0.0125$ & 4.01e$-$03& 3.39e$-$03 & $0.99$ & $1.01$ & 39922 s & 8532 s  &4.7\\
\hline
\end{tabular}
\caption{Convergence study for an ellipsoid identity map using the $\text{LSM}_{\M}^{\N}$  and the $\text{CPM}_{\M}^{\N}$ with $t_f = 0.01.$}
\label{tab:ls_cp_ellipsoid_ellipsoid}
\end{table}

\begin{table}
\centering
\begin{tabular}{p{1.25cm} K{1.75cm} K{1.75cm} K{1.5cm} K{1.5cm} K{1.75cm} K{1.75cm} c}
\hline
$\Delta x$ &\multicolumn{2}{c}{Error} & \multicolumn{2}{c}{Conv. rate} & \multicolumn{2}{c}{Comp. time} & Speedup\\
\hline
& $\text{LSM}_{\M}^{\N}$ & $\text{CPM}_{\M}^{\N}$ & $\text{LSM}_{\M}^{\N}$ &$\text{CPM}_{\M}^{\N}$  & $\text{LSM}_{\M}^{\N}$ & $\text{CPM}_{\M}^{\N}$& \\
\hline
$0.2$ &7.33e$-$02 & 5.57e$-$02 & & &1.056 s &  0.058 s & 18.1\\
$0.1$ &3.44e$-$02 &  2.71e$-$02 &$1.09$ &$1.04$ & 6.143 s& 0.689 s & 8.9 \\
$0.05$ & 1.68e$-$02& 1.38e$-$02 & $1.03$ & $0.98$ & 77.47 s& 11.49 s & 6.7\\
$0.025$ &8.06e$-$03 & 6.90e$-$03 & $1.06$& $1.00$ & 1303  s& 220.6 s & 5.9\\
$0.0125$ & 4.07e$-$03 & 3.48e$-$03 &$0.99$ & $0.99$ & 23916 s& 3134 s & 7.6\\
\hline
\end{tabular}
\caption{Convergence study for a torus identity map using the $\text{LSM}_{\M}^{\N}$  and the $\text{CPM}_{\M}^{\N}$ with $t_f = 0.01.$ }
\label{tab:ls_cp_torus_torus}
\end{table}

With respect to accuracy we observe that the $\text{CPM}_{\M}^{\N}$ is the decisive winner. The convergence rates of the $\text{CPM}_{\M}^{\N}$ and the $\text{LSM}_{\M}^{\N}$ are similar for all the manifolds: Tables~\ref{tab:ls_cp_sphere_sphere}-\ref{tab:ls_cp_torus_torus} confirm a first-order convergence rate for both the $\text{CPM}_{\M}^{\N}$ and the $\text{LSM}_{\M}^{\N}$. However, note that the width of the computational band for the $\text{LSM}_{\M}^{\N}$ had to be three times that of the $\text{CPM}_{\M}^{\N}$ to obtain first-order convergence. If a smaller band width is used, the $\text{LSM}_{\M}^{\N}$ has poor convergence or does not converge at all. This large band width requirement is a contributing factor for the high computational work of the $\text{LSM}_{\M}^{\N}$.

The computation time of each example is also given in Tables~\ref{tab:ls_cp_sphere_sphere}-\ref{tab:ls_cp_torus_torus}. As expected, the $\text{CPM}_{\M}^{\N}$ is faster than the $\text{LSM}_{\M}^{\N}$ in all experiments. The $\text{CPM}_{\M}^{\N}$ speedup roughly ranges from a factor of 5 to 10. Recall that our  implementation of the $\text{LSM}_{\M}^{\N}$ uses the closest point extension as a replacement for the extension evolution~(\ref{reext}), which further improves the efficiency of the $\text{LSM}_{\M}^{\N}$ over the original implementation~\cite{Memoli2004}. For convenience, the rightmost columns of Tables~\ref{tab:ls_cp_sphere_sphere}-\ref{tab:ls_cp_torus_torus} give the $\text{CPM}_{\M}^{\N}$ speedup, (Comp. time using the $\text{LSM}_{\M}^{\N}$)/(Comp. time using the $\text{CPM}_{\M}^{\N}$).

%%%%%%%%%%%%%%%%%%%%%%%%%%%%%%%%%%%%%%%%%%%%%
%%%%%%%%%%%%%%%%%%%%%%%%%%%%%%%%%%%%%%%%%%%%%

\section{Numerical Results}
\label{sec:num_res}
There are numerous areas of application for harmonic maps and general manifold mappings. Some of these, such as direct cortical mapping~\cite{Shi2007,Shi2009}, are interested in the values of the map $\u.$ Other applications are primarily visual. In this section, we highlight the behaviour and performance of our method with three visual applications.  Section~\ref{sec:app_texture_map} denoises texture maps following an idea from M{\'e}moli et al.~\cite{Memoli2004}. Section~\ref{sec:rand_map} diffuses a random map between two general manifolds to a point. Finally, colour image enhancement via chromaticity diffusion~\cite{Tang2001} is performed in Section~\ref{sec:colour_im_denoise}. 

\subsection{Diffusion of noisy texture maps}
\label{sec:app_texture_map}
Our first numerical experiments diffuse noisy texture maps.  Since texture maps give a means to visualize the map $\u^k,$ they are helpful for providing intuition and insight into our algorithms.  

To begin, a texture map ${\bf T}$ is created using the ideas of Zigelman et al.~\cite{Zigelman2002}. The map ${\bf T}$ is inverted to yield a map ${\bf w}({\bf x}) : D \rightarrow \N$ from the planar image domain $D$ to the manifold $\N.$ A noisy map is created by adding a normally distributed random map ${\bf r}({\bf x} ): D \rightarrow \mathbb{R}^3$ to ${\bf w}.$ The sum of ${\bf r}$ and ${\bf w}$ is generally not on $\N$ so this summation is followed by a projection step onto $\N.$ This gives a noisy map $\u^0(\x) : D \rightarrow \N$ from the image domain to the manifold $\N$ defined by 
\begin{equation}
\u^0(\x) = \cp_{\N} ({\bf w}(\x) + {\bf r}(\x)).
\label{init_noisy_map}
\end{equation}

The gradient descent equations~(\ref{proj_graddes}) are expected to diffuse $\u^0$ analogous to $H^1$-regularization of a planar image. It is important to recognize that the initial map $\u^0$ is evolved, not the colour values of the image. Our purpose for placing an image on $\N$ is to visualize and compare the initial map $\u^0$ and the harmonic map computed by the $\text{CPM}_{\M}^{\N}$ (Algorithm~\ref{alg:cpmmm}). 

The $\text{CPM}_{\M}^{\N}$  is applied with $\u^0$ as the initial condition. Numerical implementation of the $\text{CPM}_{\M}^{\N}$ is nearly identical to the $\text{CPM}_{\M}$ for unconstrained PDEs on manifolds~\cite{Ruuth2008}. The sole difference is the need to  evaluate the closest points of $\u_{\text{ext}}(\x)$ on $\N$ via $\cp_{\N}(\u_{\text{ext}}(\x)).$  Our codes are all straightforward modifications of existing $\text{CPM}_{\M}$ software~\cite{cp_matrices}. 

In all examples, the heat equation in the $\text{CPM}_{\M}^{\N}$ is discretized by second-order centred differences in space and forward Euler in time. A time step-size of $\Delta t = 0.1 \Delta x^2$ is used. Textures on manifolds are visualized using the \texttt{patch} command in M\textsc{atlab}; a triangulation of the manifold is used in this final visualization step.

\subsubsection{Harmonic maps from a plane to $S^2$}
 In our first example, we compute the harmonic map from $\M \subset \mathbb{R}^2$ to $\N = S^2$ and conduct a numerical convergence study. Each time step of the $\text{CPM}_{\M}^{\N}$ performs heat flow in the Euclidean space $\M \subset \mathbb{R}^2,$ followed by a projection onto $\N$ using $\cp_{\N}.$ Since $\M \subset \mathbb{R}^2$ the $\text{CPM}_{\M}$ is not used for the first step of the $\text{CPM}_{\M}^{\N}.$ Instead, the heat equation is directly solved on the plane $\M,$ while imposing Neumann boundary conditions.
 
 The closest point function for $\N=S^2$ has the explicit formula $$\cp_{S^2}(\mathbf{z}) = \frac{\mathbf{z}}{\|\mathbf{z}\|_2}.$$ Therefore, in this example, one evolves the heat equation on $\Omega_c =[-1,1]^2 \subset\mathbb{R}^2$ over a time $\Delta t$ starting from $\tilde{\u}(\x,0) =\u^k(\x),$ to give $\tilde{\u}(\x, \Delta t).$ Then the solution at time $t^{k+1}$ is simply $$\u^{k+1}(\x) = \frac{\tilde{\u}(\x, \Delta t)}{\|\tilde{\u}(\x, \Delta t)\|_2}.$$ Note that there is no $\u_{\text{ext}}(\x)$ in this example since $\M$ is a plane.
 
 We apply a spatial discretization step-size of  $\Delta x = 0.005$ and evolve for 300 time steps. The random map, ${\bf r},$ is constructed using the method described in Section~\ref{subsec:ident} with $\alpha = 0.05.$  Figure~\ref{fig:res_sphere} shows the noisy map $\u^0(\x)$ (left column) and our computed harmonic map $\u^k(\x)$ (right column) at two viewing angles. 
   %%%%%%%%%%%%%%%%%%% SPHERE %%%%%%%%%%%%%%%%%%%%%%%%%
\begin{figure}
\centering
\begin{tabular}{cc}
 \includegraphics[width=1.5in]{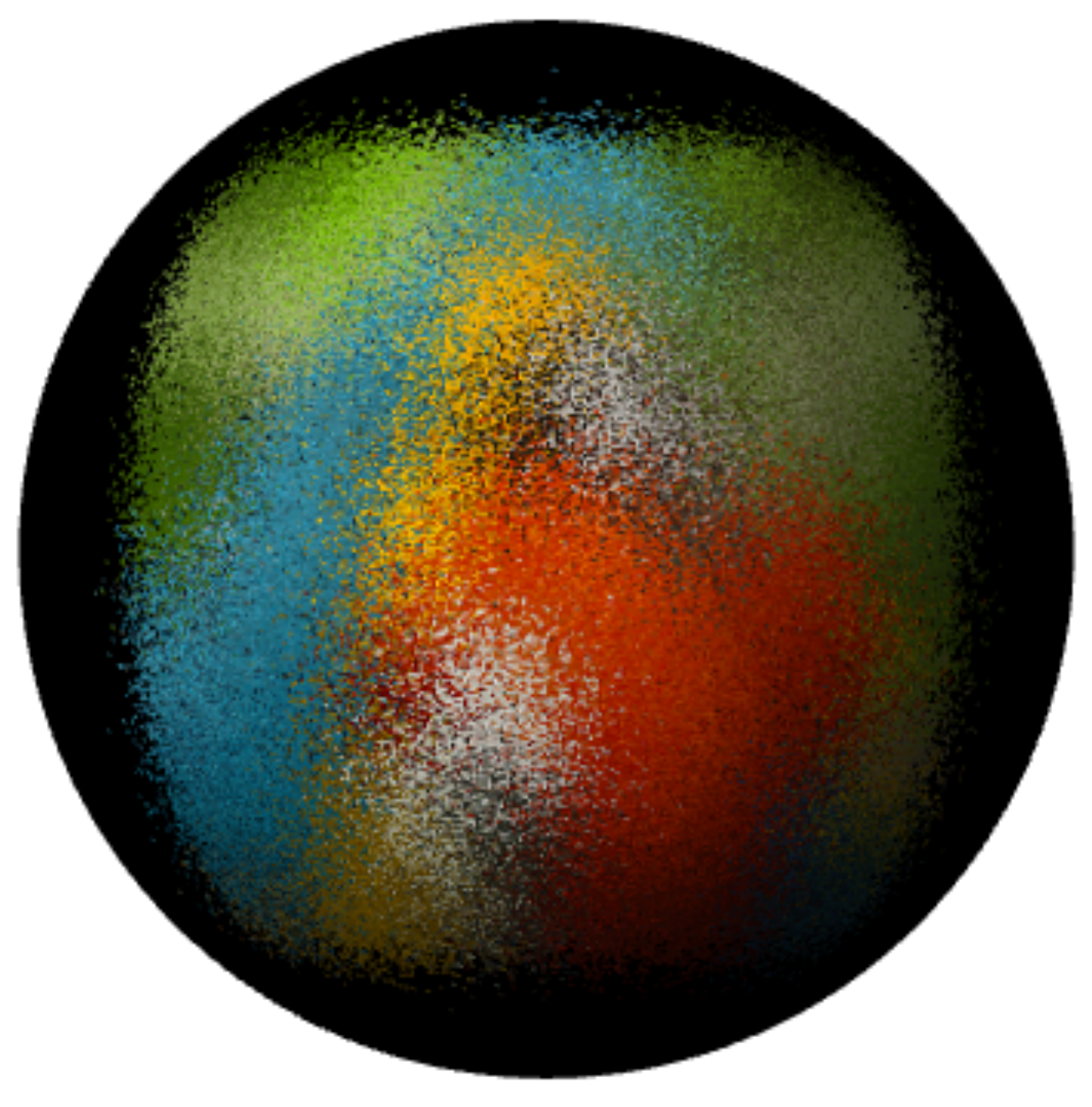}  & \includegraphics[width=1.5in]{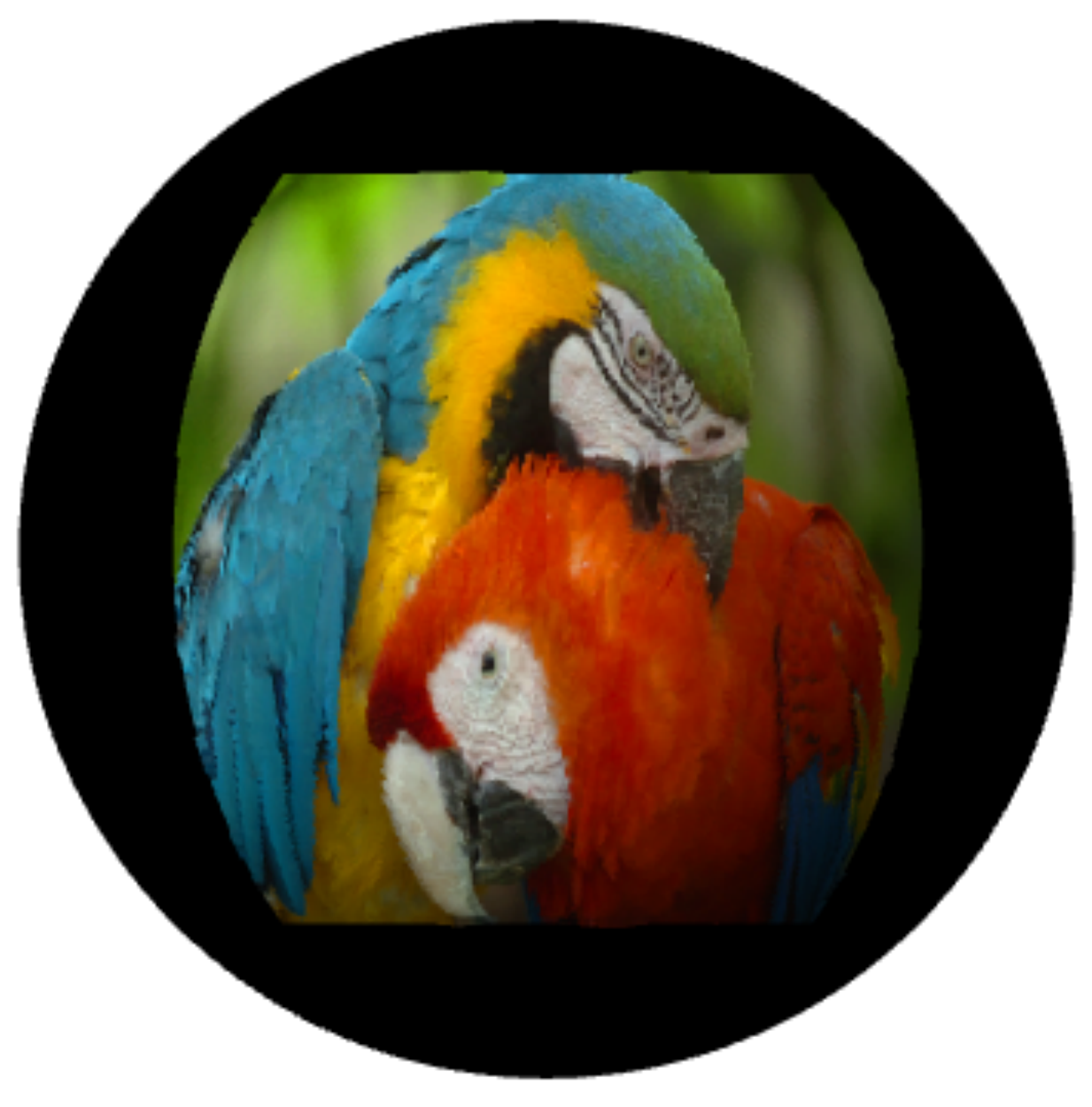}\\
 \includegraphics[width=1.5in]{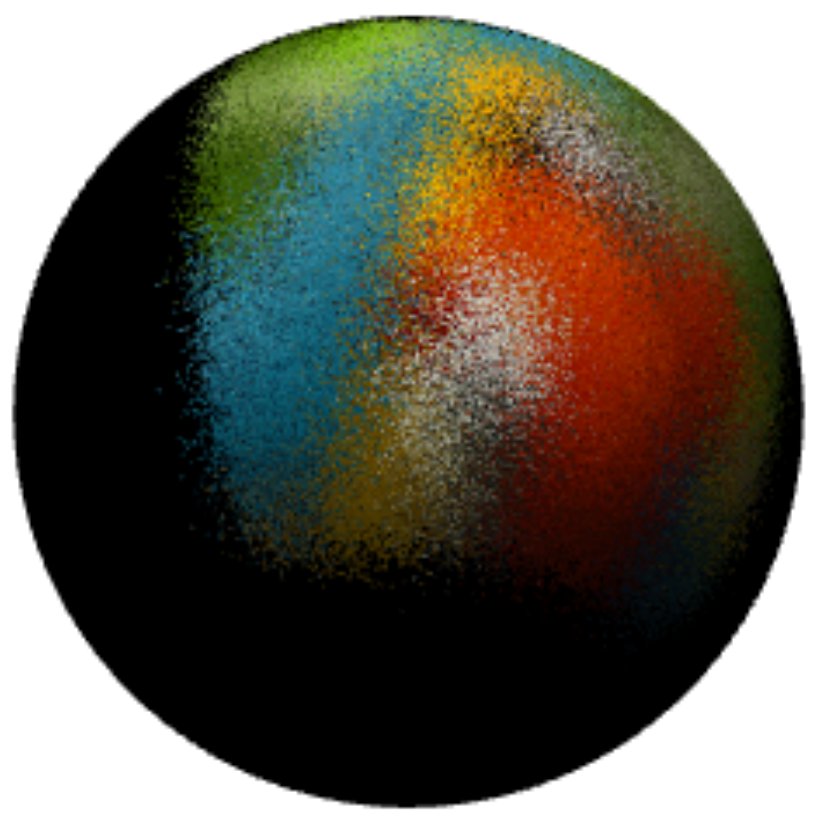}  & \includegraphics[width=1.5in]{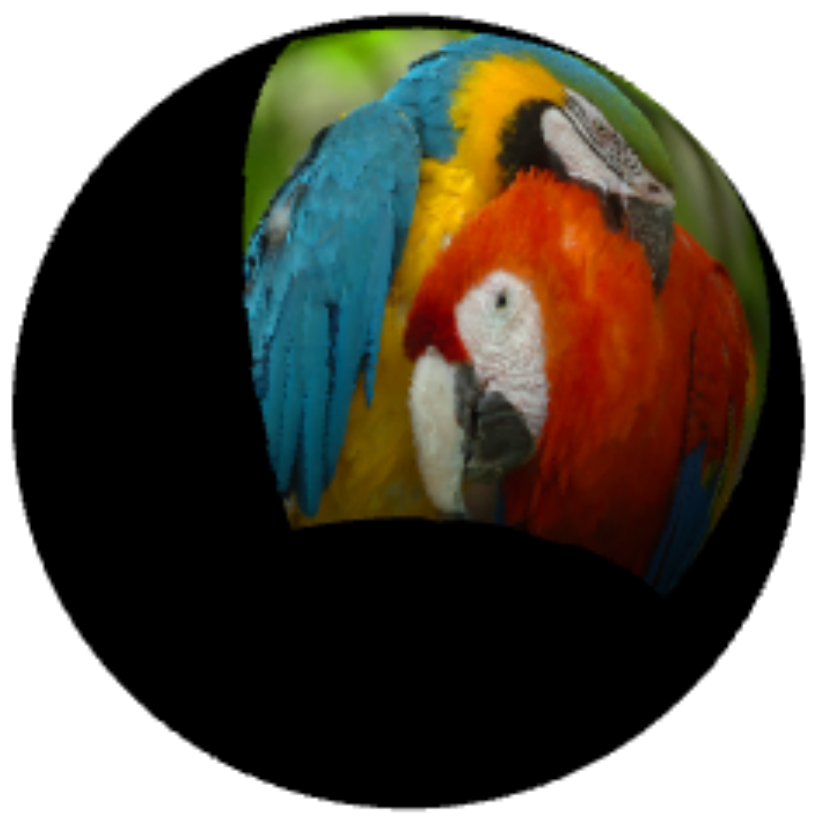}\\
\end{tabular}
\caption{A noisy map from $\M\subset \mathbb{R}^2$  onto the unit sphere (left column) is denoised via the computation of a harmonic map (right column).}
\label{fig:res_sphere}
\end{figure}
We see that the harmonic map from $\M \subset \mathbb{R}^2$ to $\N=S^2$ computed using the $\text{CPM}_{\M}^{\N}$ is a denoised version of the initial map $\u^0(\x).$ The planar image of parrots is courtesy of~\cite{parrots}.

We conclude this subsection with a convergence study for $\M = [-1,1]^2 \subset \mathbb{R}^2$ and $\N = S^2.$ There is no analytical solution for this example so we compare results against a reference solution, $\u_{\text{ref}},$ that was computed using $\Delta x = 0.0015625.$ The error in $\u^k$ (when compared to $\u_{\text{ref}}$) is computed using several $\Delta x$ values and at a final time $t_f= 0.01.$ The maximum Euclidean distance between $\u^k$ and $\u_{\text{ref}}$ over nodes $\mathbf{z}_j \in \M$ is used as the error estimate, i.e., $$ \text{(Error Est.)} = \hspace{0.01cm} \max_{\mathbf{z}_j} \left\|  \u^k(\mathbf{z}_j) -\u_{\text{ref}}(\mathbf{z}_j) \right\|_2 .$$ Averaging over 96 realizations to account for the random initial map, we observe first-order convergence. See Table~\ref{convTab} for the results.
\begin{table}
\centering
\begin{tabular}{p{6.75cm} p{6.25cm} c}
\hline
$\Delta x$ & Error Est.  & Conv. rate\\
\hline
$0.1$ & 5.41e$-$02 & \\
$0.05$ & 2.82e$-$02 & $0.93$ \\
$0.025$ & 1.47e$-$02 & $0.94$ \\
$0.0125$ & 7.25e$-$03 & $1.02$ \\
$0.00625$ & 3.45e$-$03 & $1.07$ \\
\hline
\end{tabular}
\caption{Convergence study of errors between a reference solution $\u_{\text{ref}}$ and a harmonic map $\u^k$ from a plane to $S^2.$ }
\label{convTab}
\end{table}

\subsubsection{Harmonic maps from a plane to Laurent's hand}
\label{sec:hand}
In our second example, a harmonic map from a source manifold $\M \subset \mathbb{R}^2$ to a hand target manifold is constructed. The hand target manifold $\N$ is ``Laurent's hand,'' a triangulated manifold available in the AIM@SHAPE repository~\cite{AimShape}. Geodesic distances needed for the texture mapping algorithm of Zigelman et al.~\cite{Zigelman2002} are computed using the method of Crane et al.~\cite{Crane2013}.

In the second step of the $\text{CPM}_{\M}^{\N},$ the closest point to $\N$ is evaluated for all the irregularly spaced target values $\tilde{\u}(\x, \Delta t).$  In our implementation, we evaluate the closest point to the triangulation for each $\tilde{\u}(\x, \Delta t)$ by a local search over the triangulation. That is, we pre-compute the closest point function $\cp_{\N}$ on a uniform grid surrounding the surface (using, e.g., \texttt{tri2cp.m}~\cite{cp_matrices}) and use it to localize the closest point search of $\tilde{\u}(\x, \Delta t)$ to $\N.$ Each $\tilde{\u}(\x, \Delta t)$ belongs to a cube defined by 8 vertices. The closest point values of these vertices yield a set $\mathcal{S}$ of up to 8 triangles. We search for the closest point of $\tilde{\u}(\x,\Delta t)$ over all triangles of $\N$ that are within a small bounding sphere for $\mathcal{S}.$

Using the $\text{CPM}_{\M}^{\N}$ with a spatial discretization step-size $\Delta x=0.005$ and 30 time steps, we compute a harmonic map starting from the initial, noisy map $\u^0.$ The random map, ${\bf r},$ is constructed using the method described in Section~\ref{subsec:ident}. In this example, however, a different scaling parameter is used in the $z$ coordinate direction than in the $x$ and $y$ directions. Specifically, values of $\alpha = 0.0025,$ $0.0025,$ and $0.001$ are taken for the $x,$ $y,$ and $z$ directions, respectively.  See Figure~\ref{fig:res_hand} for two viewing angles of the initial map $\u^0$ (left column) and the $\text{CPM}_{\M}^{\N}$ computed harmonic map $\u^k$ (right column).
\begin{figure}
\centering
\begin{tabular}{cc}
 \includegraphics[width=1.2in]{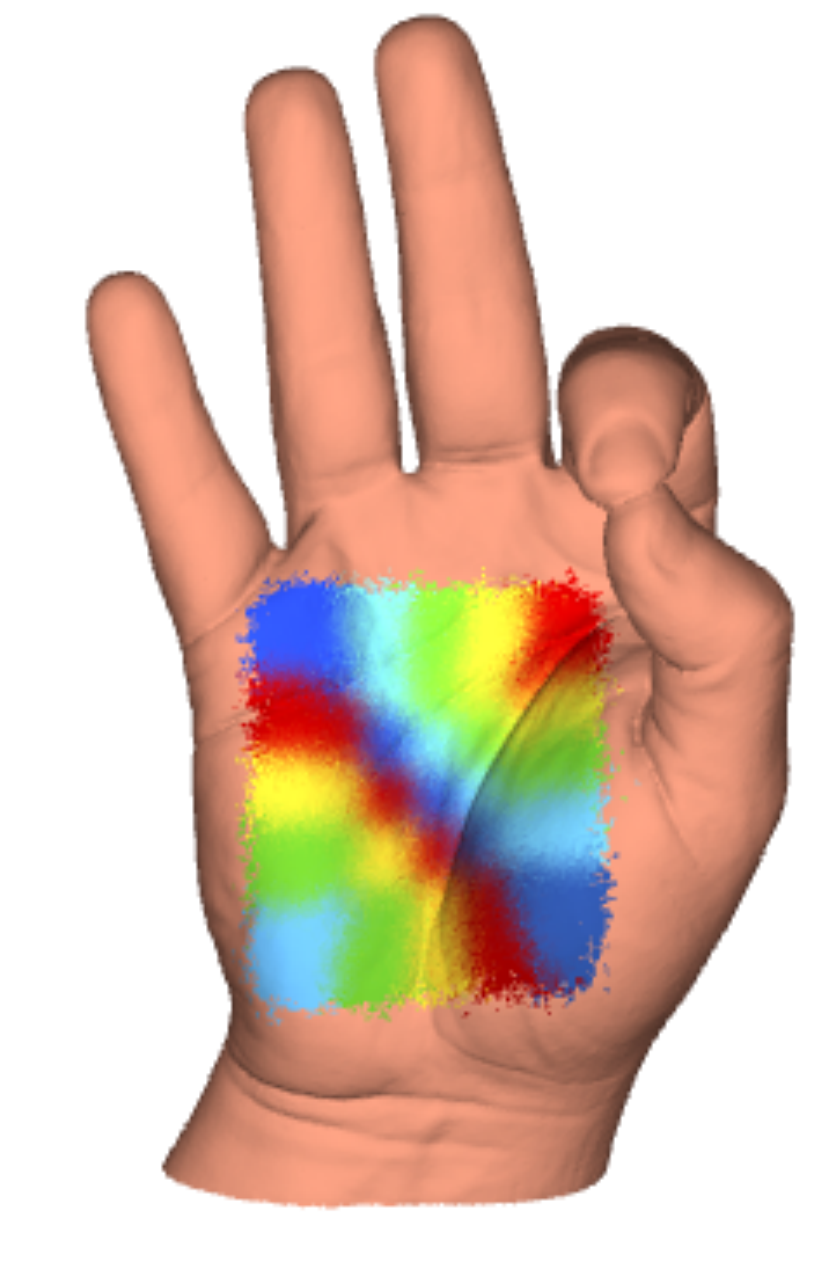}  & \includegraphics[width=1.2in]{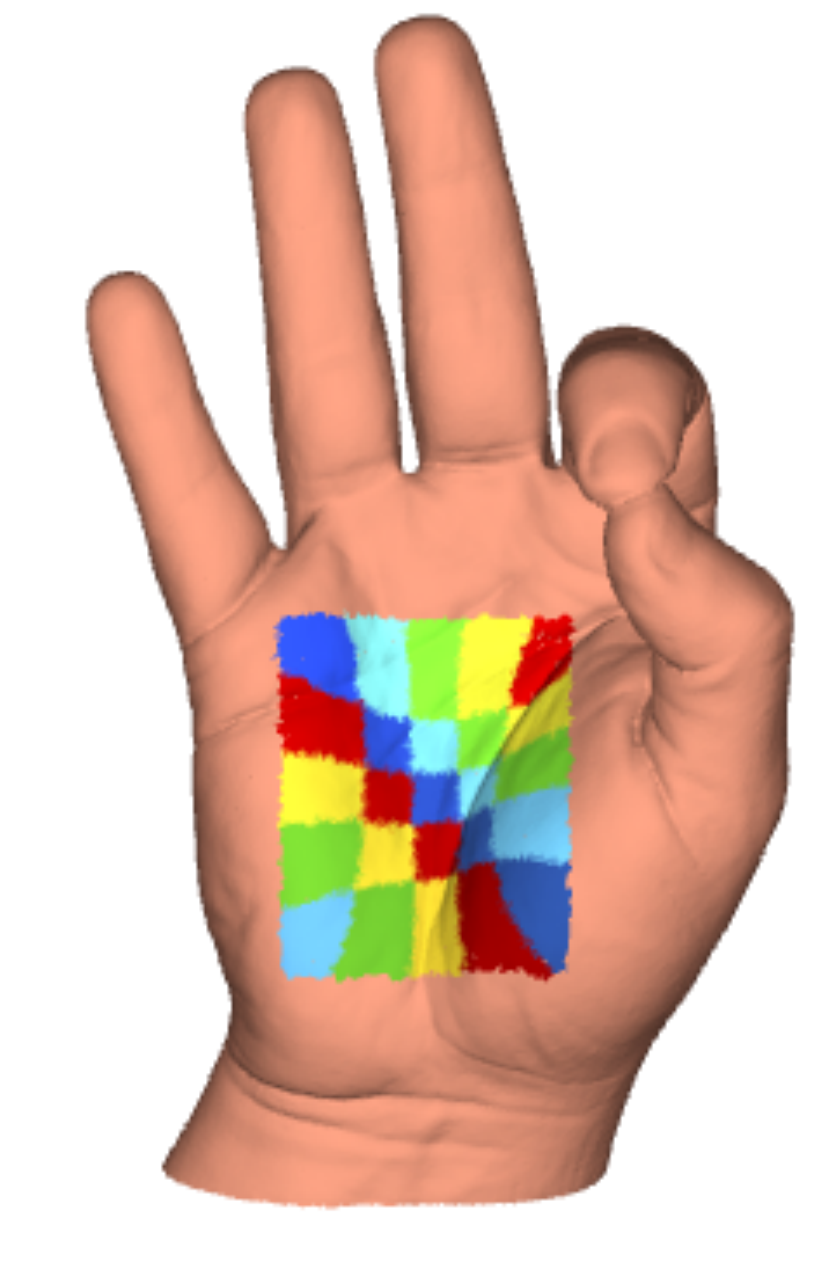} \\
    \includegraphics[width=1.7in]{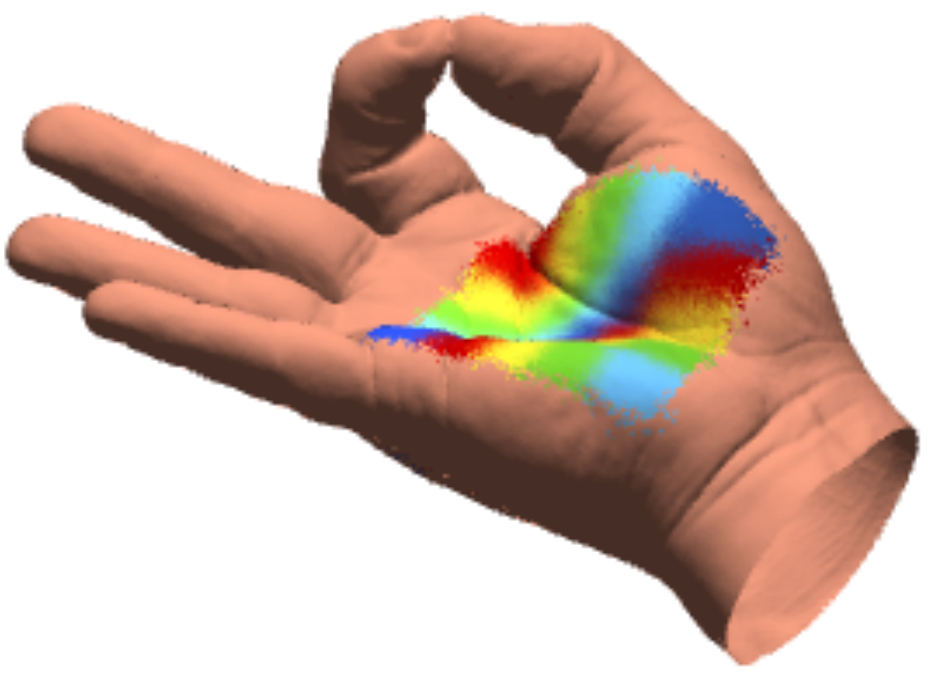}  & \includegraphics[width=1.7in]{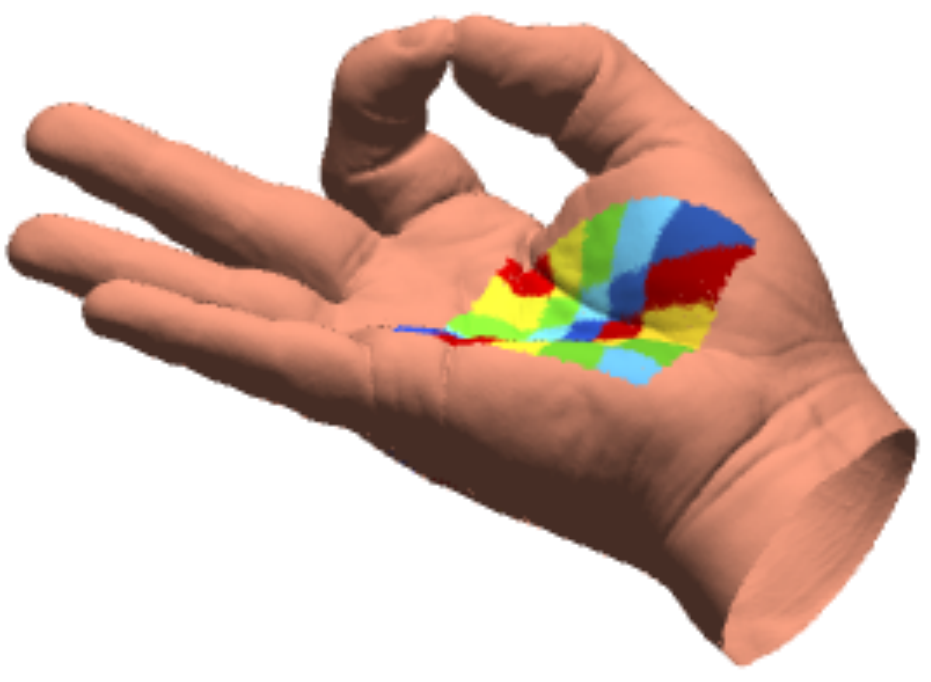} \\
\end{tabular}
\caption{A noisy map from $\M\subset \mathbb{R}^2$ onto a hand target manifold (left column) is denoised via the computation of a harmonic map (right column).}
\label{fig:res_hand}
\end{figure}
As in our previous example, the computed harmonic map is much less noisy than the initial map $\u^0.$ Note that our target $\N$ is a complex, open triangulated manifold. Open and closed manifolds are handled identically using the $\text{CPM}_{\M}^{\N}.$ The $\text{LSM}_{\M}^{\N}$ implementation would be more challenging since level set functions are only natural for closed manifolds.

\subsubsection{Harmonic maps from a cylinder to a submanifold of $S^2$}
An example of computing harmonic maps between two different curved manifolds is now considered. Specifically, we compute from a cylinder $\M$ to a portion of the unit sphere $\N.$  We take as our source a unit radius cylinder with $z \in [-2,2]$ and no top or bottom.

An image is placed on the surface of the cylinder by a simple change of variables; intuitively, the planar image is rolled into the cylinder. To accomplish this, first scale the rectangular image so that  $x \in [-\pi, \pi]$ and $y \in [-2,2].$ Next, set the angle $\theta$ and height $z$ in cylindrical coordinates equal to the $x$ and $y$ coordinates of the image, respectively. The point $\x$ is generally not a pixel location on the cylinder, so linear interpolation is used to obtain the colour values at $\x.$ 

The construction of the initial map for this example is as follows. A map $\mathbf{w}$ is defined from the cylinder to the corresponding portion of the sphere using the closest point function $\mathbf{w}(\x) = \cp_{S^2}(\x).$  The colour at $\x$ is assigned to the corresponding point $\mathbf{w}(\x)$ to place an image on the sphere.  As in~(\ref{init_noisy_map}), the initial, noisy map is formed by adding noise and projecting onto the target manifold, $$\u^0(\x) = \cp_{S^2}( \mathbf{w}(\x) + \mathbf{r}(\x)).$$ 

Note that since the cylinder is restricted to $z\in [-2,2],$ the northern and southern portions of the sphere do not appear in our map. As a consequence, our target manifold $\N$ is chosen to be an open manifold in our implementation of the $\text{CPM}_{\M}^{\N}.$ 

Consider now removing noise in the map by computing the harmonic map from the cylinder to the restricted sphere using the $\text{CPM}_{\M}^{\N}.$ A band around the surface of the cylinder serves as the computational domain $\Omega_c.$ A spatial discretization step-size of $\Delta x = 0.00625$ and 300 time steps are used. The random map, ${\bf r},$ is constructed using the method described in Section~\ref{subsec:ident} with $\alpha = 0.075.$ Homogeneous Neumann boundary conditions are automatically applied by the $\text{CPM}_{\M}^{\N}$ on the boundaries of the cylinder (see Section~\ref{subsec:cpmmm}). We display the initial, noisy map (left column) and the computed harmonic map (right column) in Figure~\ref{fig:cyl_sph} at two viewing angles.  
\begin{figure}
\centering
\begin{tabular}{cc}
    \includegraphics[width=1.6in]{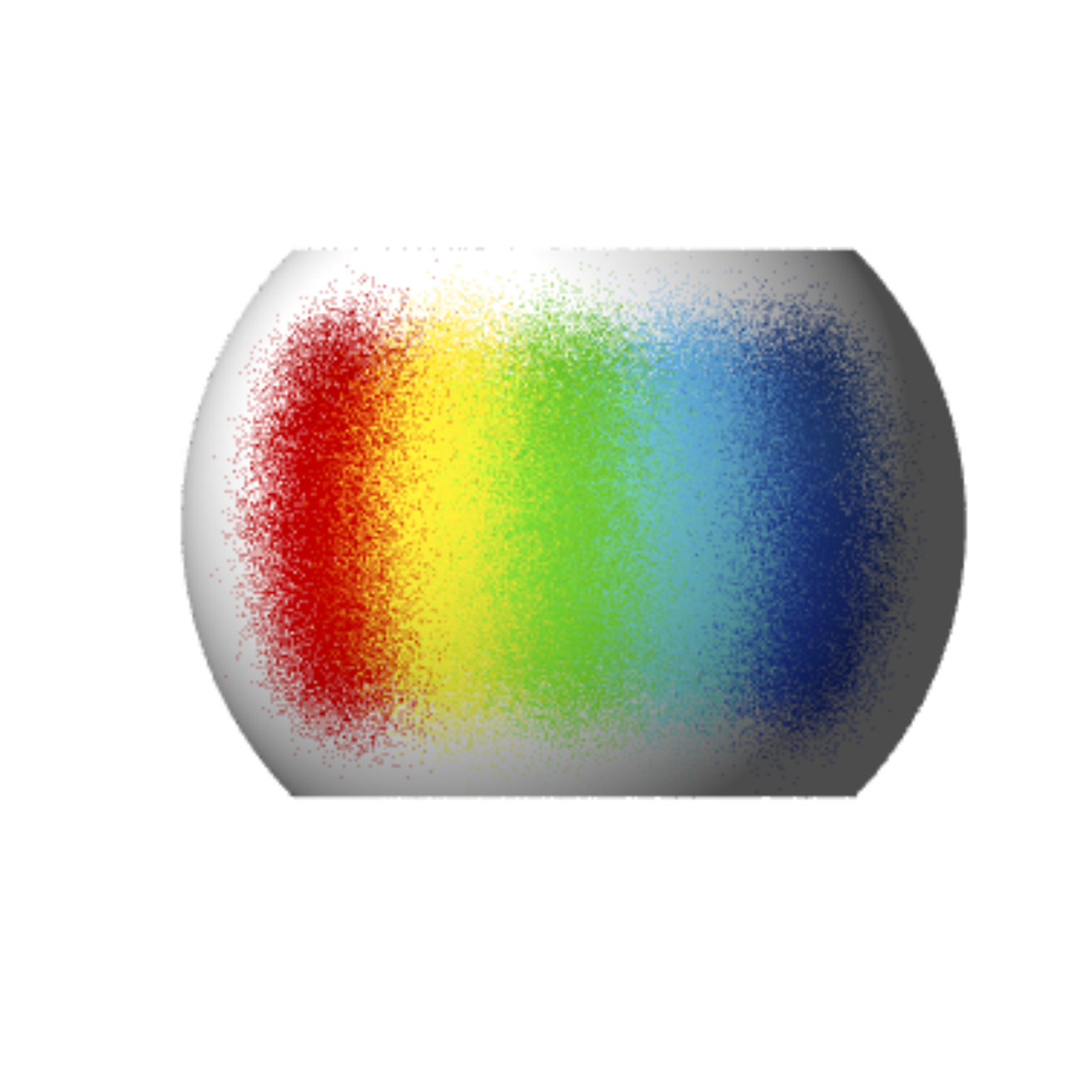}  & \includegraphics[width=1.6in]{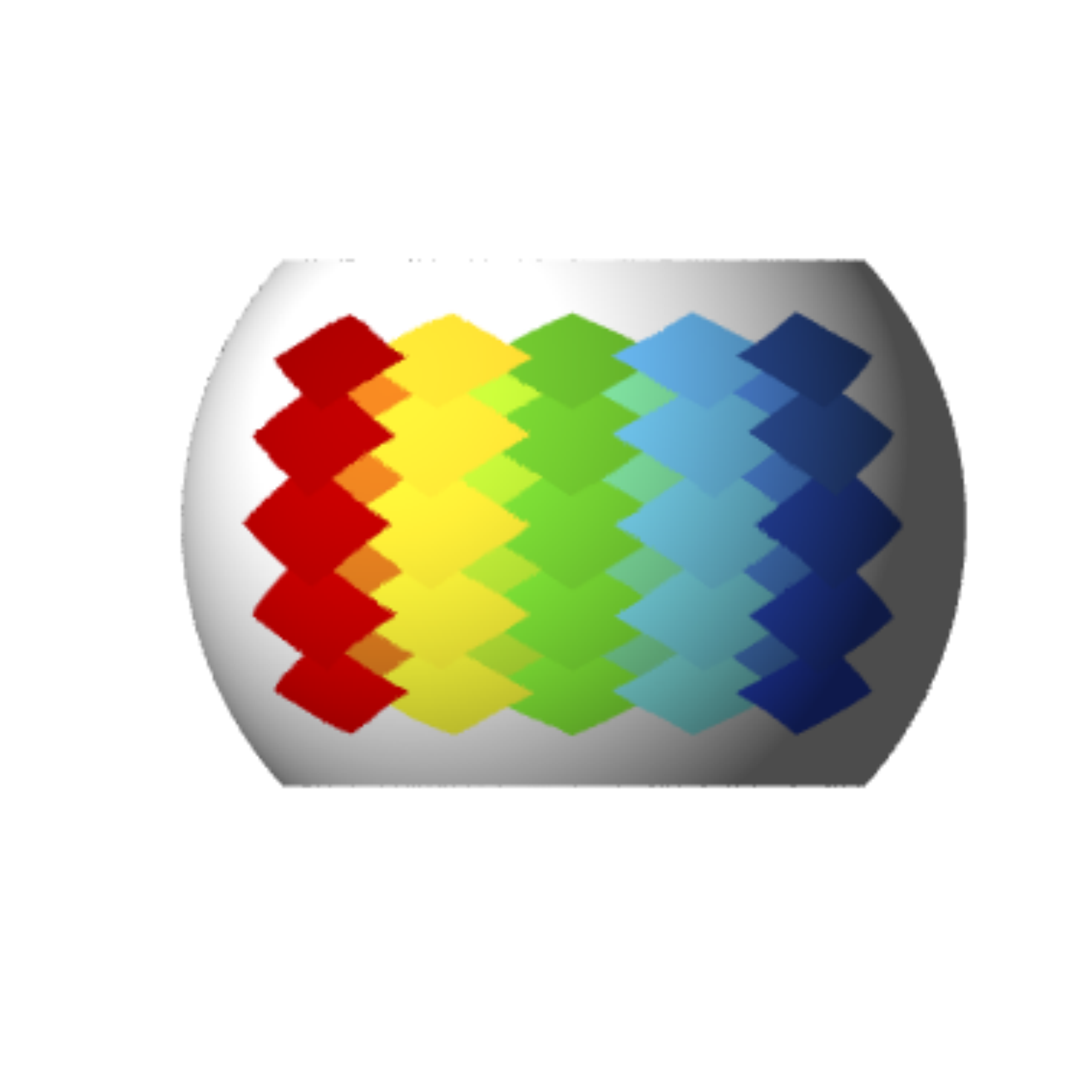} \\
    \includegraphics[width=1.6in]{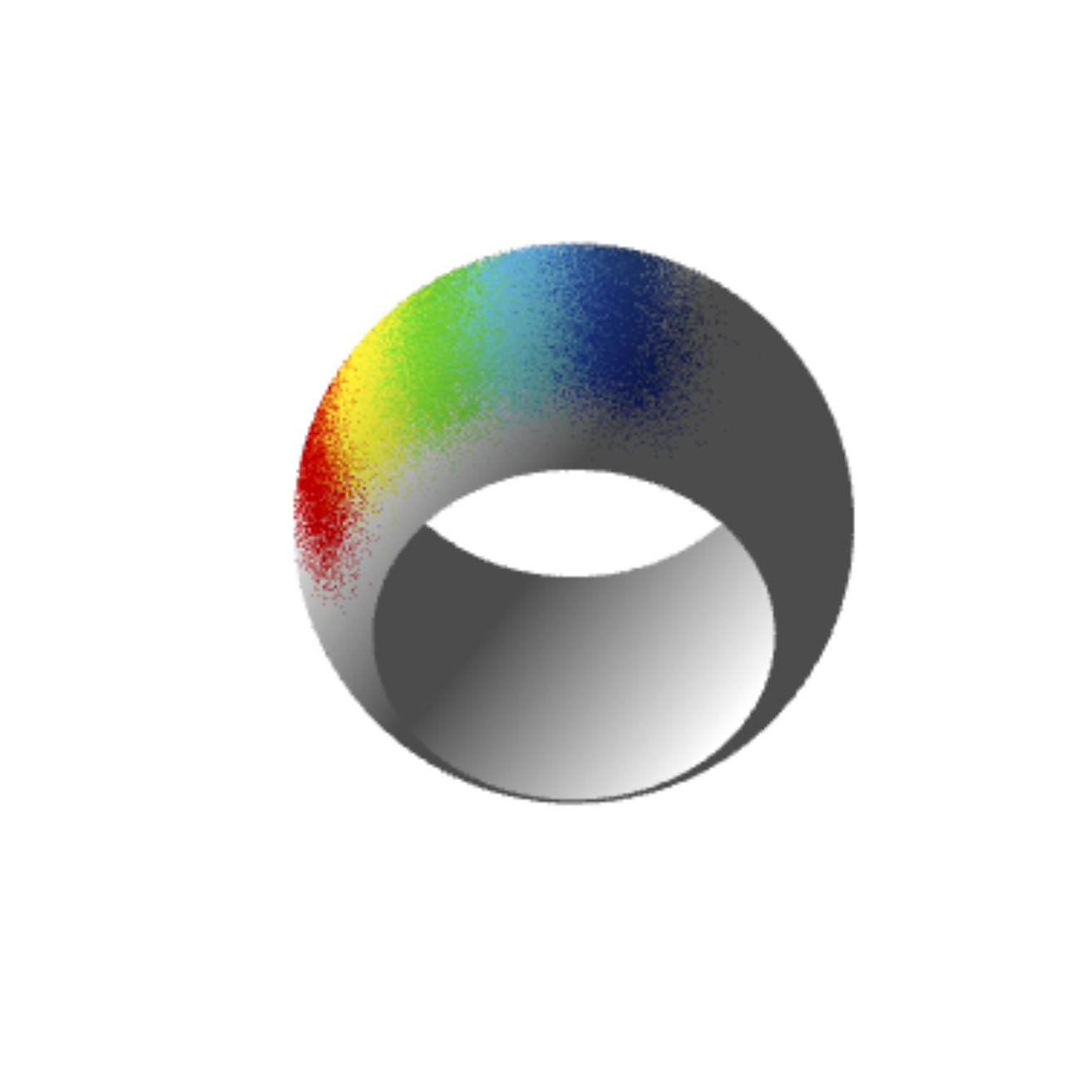}  & \includegraphics[width=1.6in]{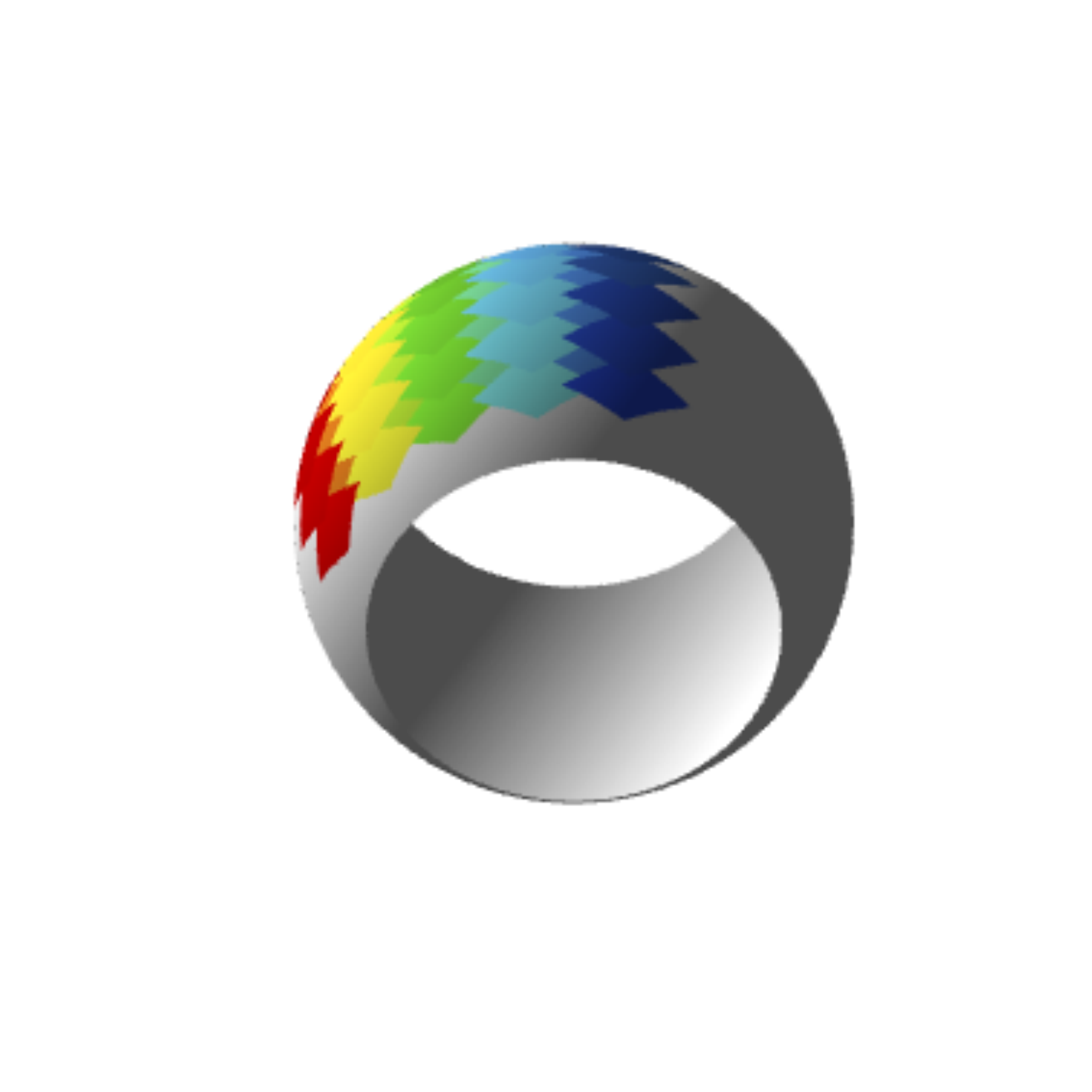} \\
\end{tabular}
\caption{Denoising of a noisy map (left column) from the unit radius cylinder ($z \in [-2,2]$) to a portion of the unit sphere. The denoised, harmonic map is visualized in the right column.}
\label{fig:cyl_sph}
\end{figure}

This example highlights the use of the $\text{CPM}_{\M}^{\N}$  for computing maps between two open, curved manifolds $\M$ and $\N.$ Frequently, methods for mapping between two general curved manifolds involve intermediate projections to a plane or sphere~\cite{Shi2007, Gibson2010}. The $\text{CPM}_{\M}^{\N}$ does not need intermediate projections, thereby eliminating a source of computational work and distortion errors. 

\subsection{Random maps from a torus to the Stanford bunny}
\label{sec:rand_map}
We now compute a harmonic map from a torus to the Stanford bunny starting from a random map. This further illustrates the computation of a harmonic map between general manifolds without resorting to intermediate projection steps. The source manifold $\M$ is a torus with minor radius 0.75 and major radius 1.25. The target manifold $\N$ is the Stanford bunny triangulation~\cite{bunny}. The Stanford bunny is an open manifold, like Laurent's hand, but has five holes instead of one. 

An initial random map from the torus to the Stanford bunny is constructed as follows. First, we choose 16 vertices $\mathbf{p}_i$ on the bunny triangulation. Then, the 240 nearest neighbours of each $\mathbf{p}_i$ are used to form 16 sets of points $\mathcal{P}_i.$ Points $\x \in \M$ are mapped to points in $\mathcal{P} = \cup \mathcal{P}_i$ by sampling uniformly with replacement using M\textsc{atlab}'s \texttt{datasample} command. The random map is evolved using the $\text{CPM}_{\M}^{\N}.$ We anticipate the evolution~(\ref{proj_graddes}) to take the initial random map to a point; see~\cite{Memoli2004} for details. 

Figure~\ref{fig:rand_map} (upper left) shows in blue where points $\x\in\M$ map onto the bunny $\N.$   
\begin{figure}
\centering
\begin{tabular}{ccc}
    \includegraphics[width=1.5in]{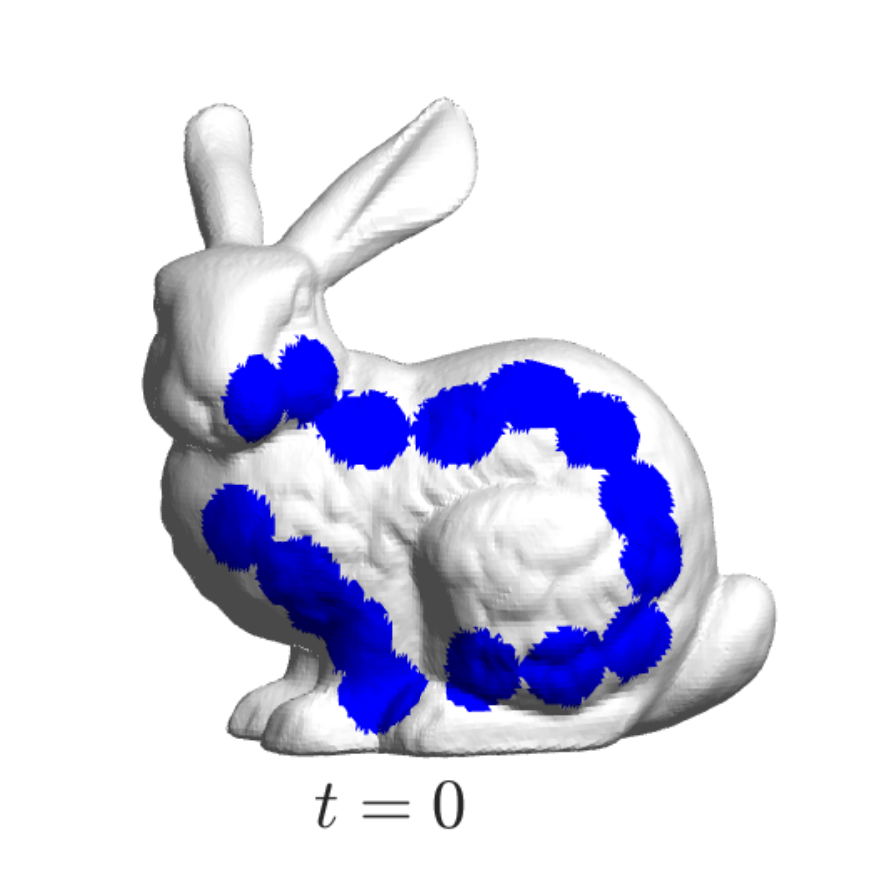}  & \includegraphics[width=1.5in]{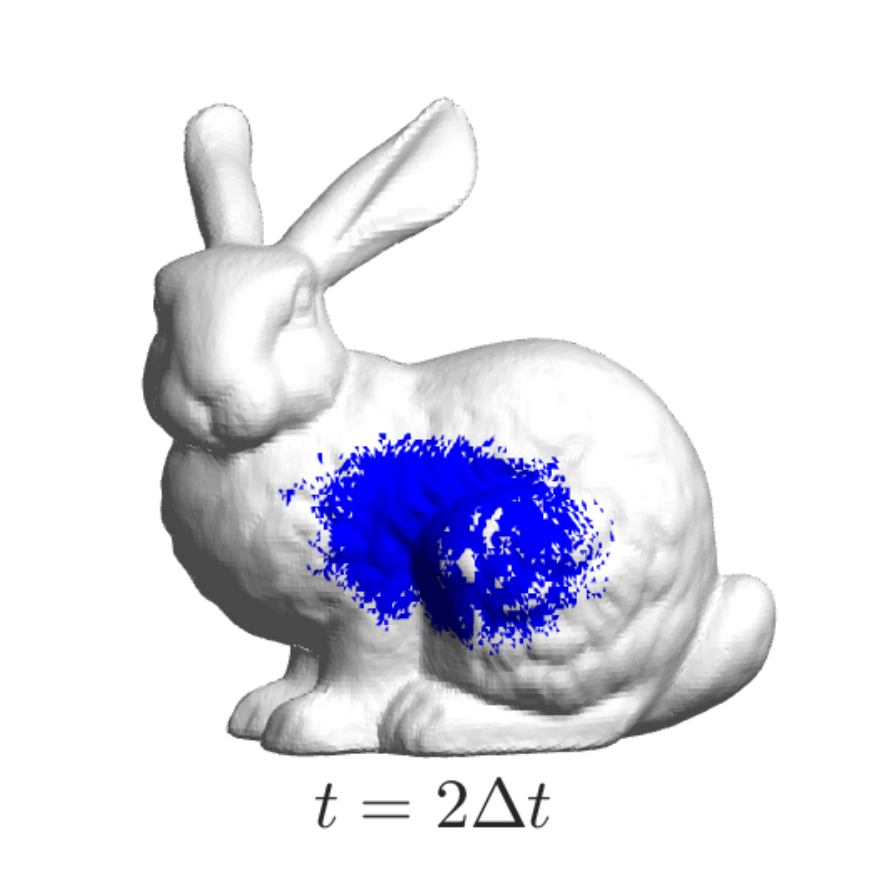} &
    \includegraphics[width=1.5in]{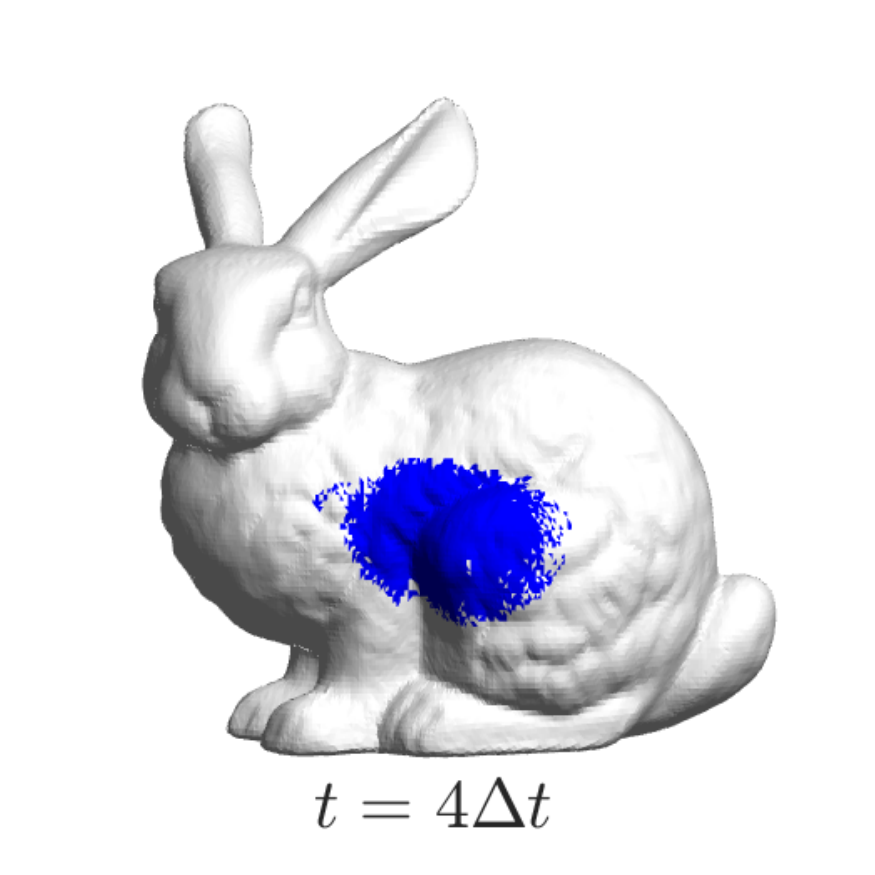}\\
   \includegraphics[width=1.5in]{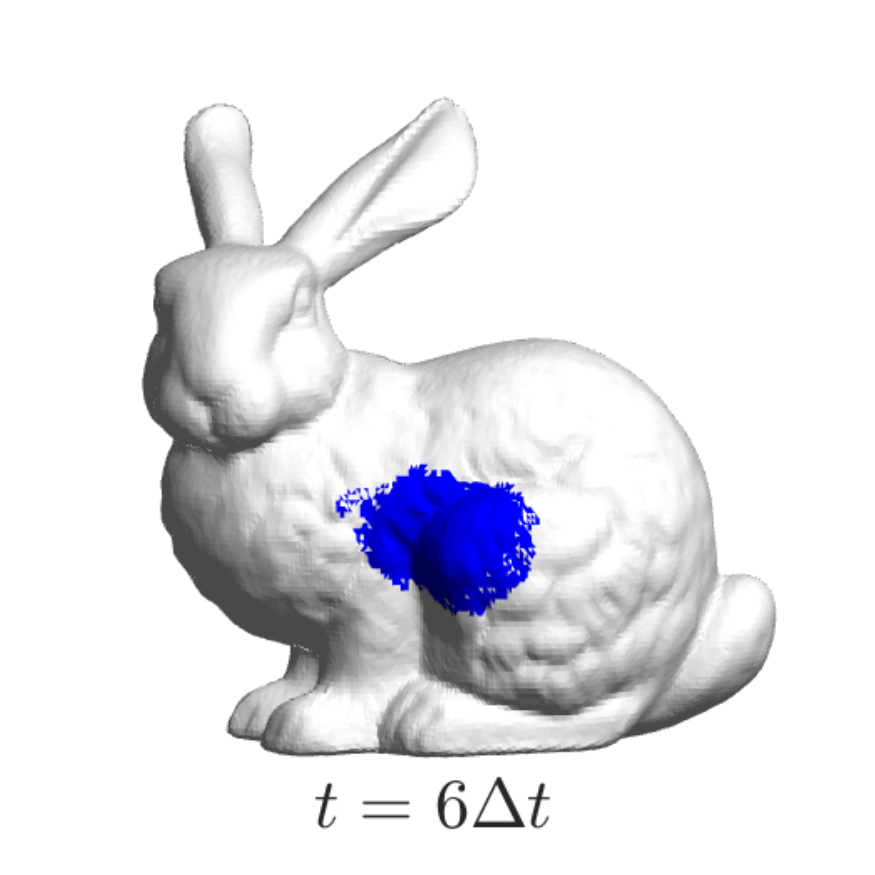}  & \includegraphics[width=1.5in]{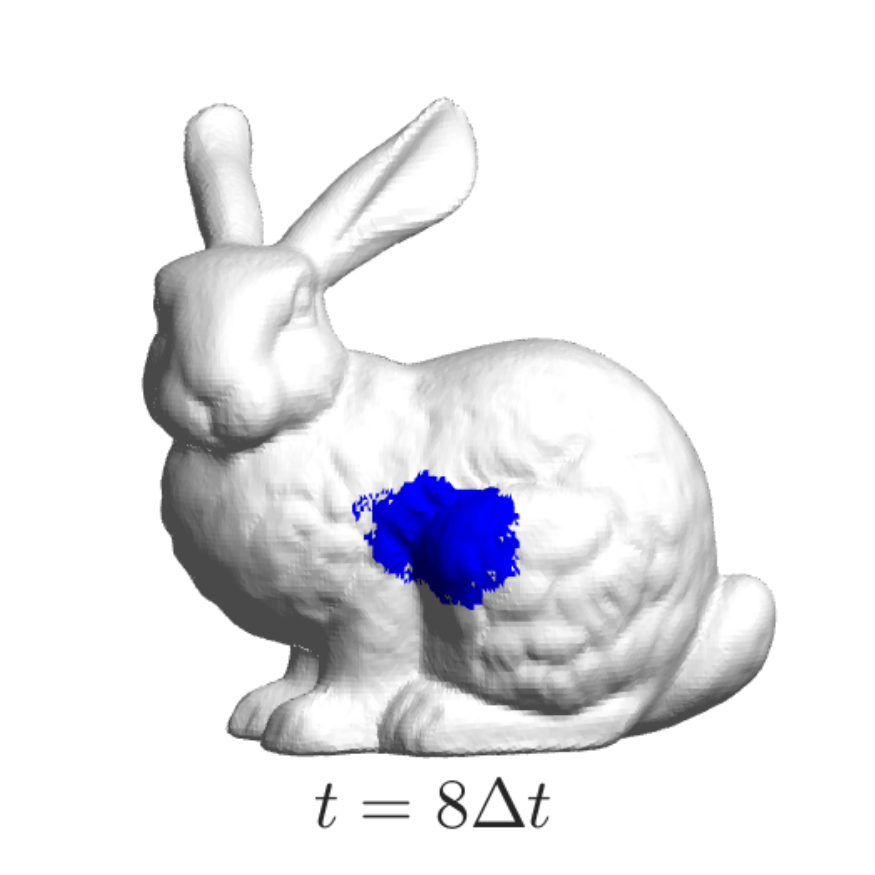} &
  \includegraphics[width=1.5in]{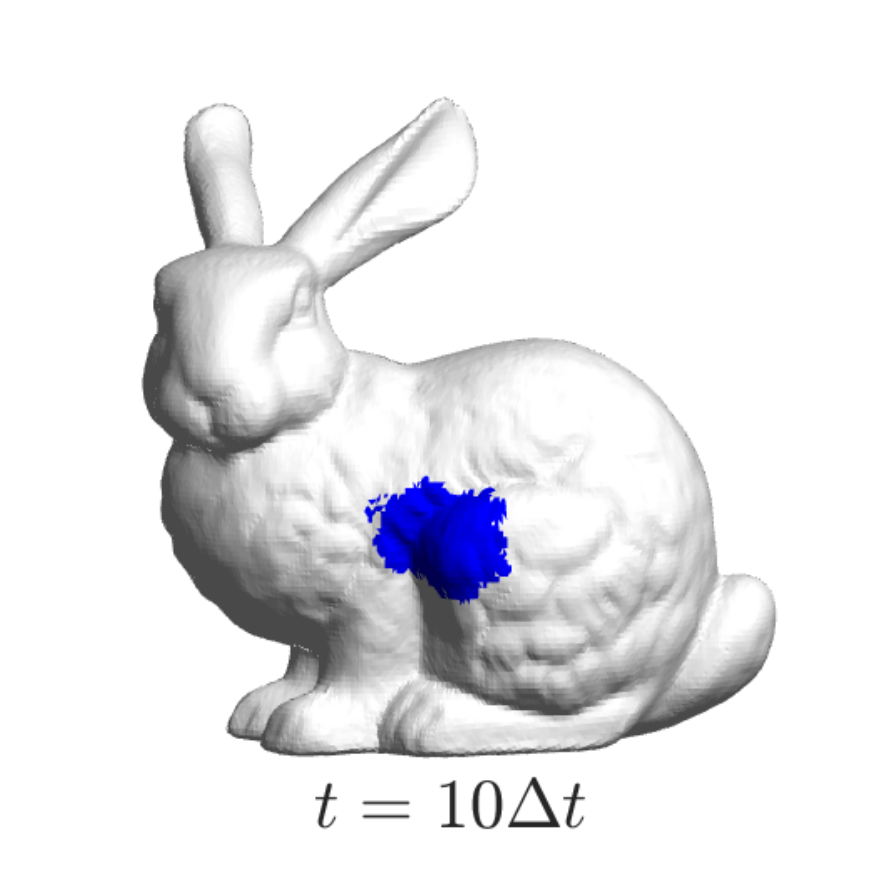}\\
      \includegraphics[width=1.5in]{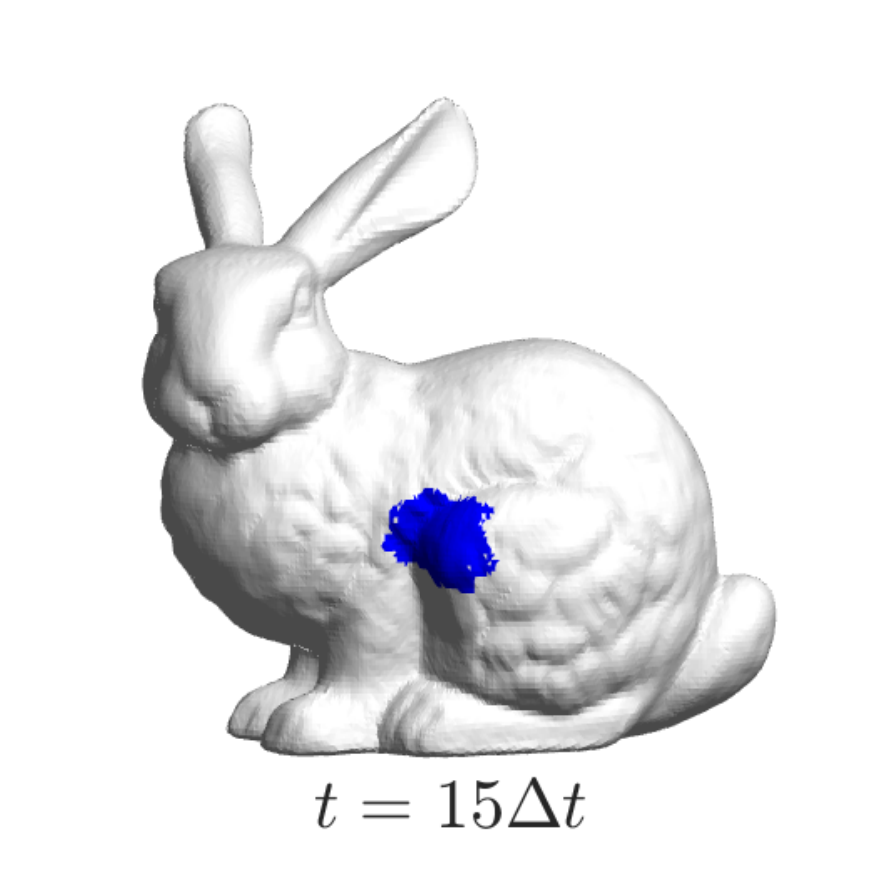}  & \includegraphics[width=1.5in]{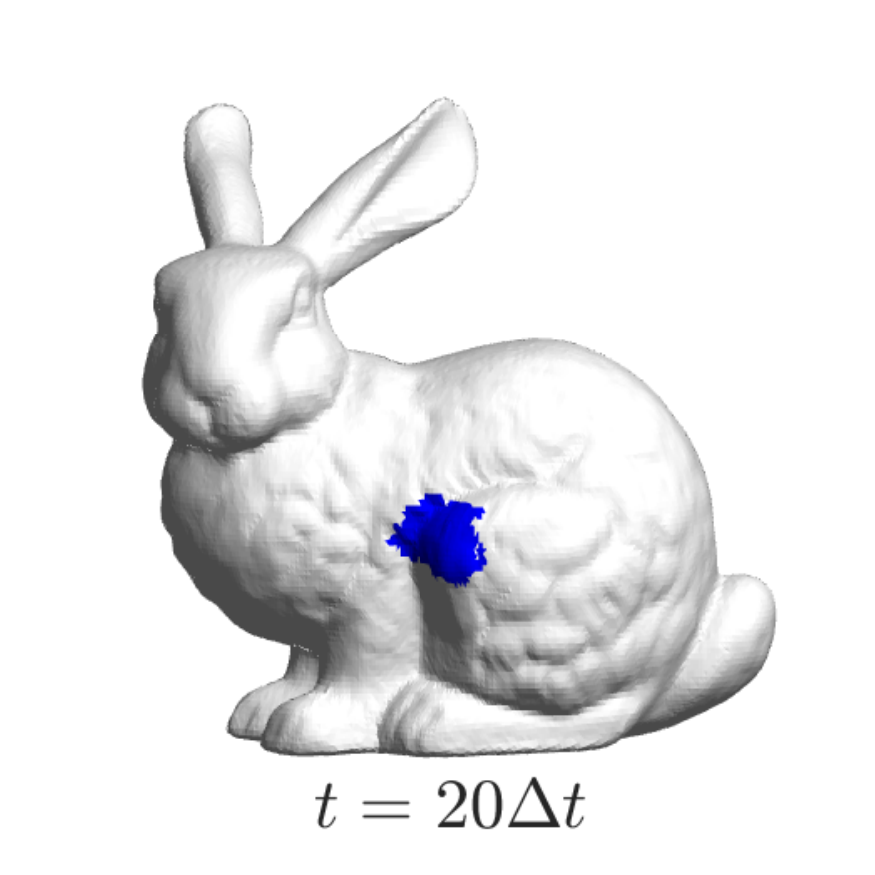} &
    \includegraphics[width=1.5in]{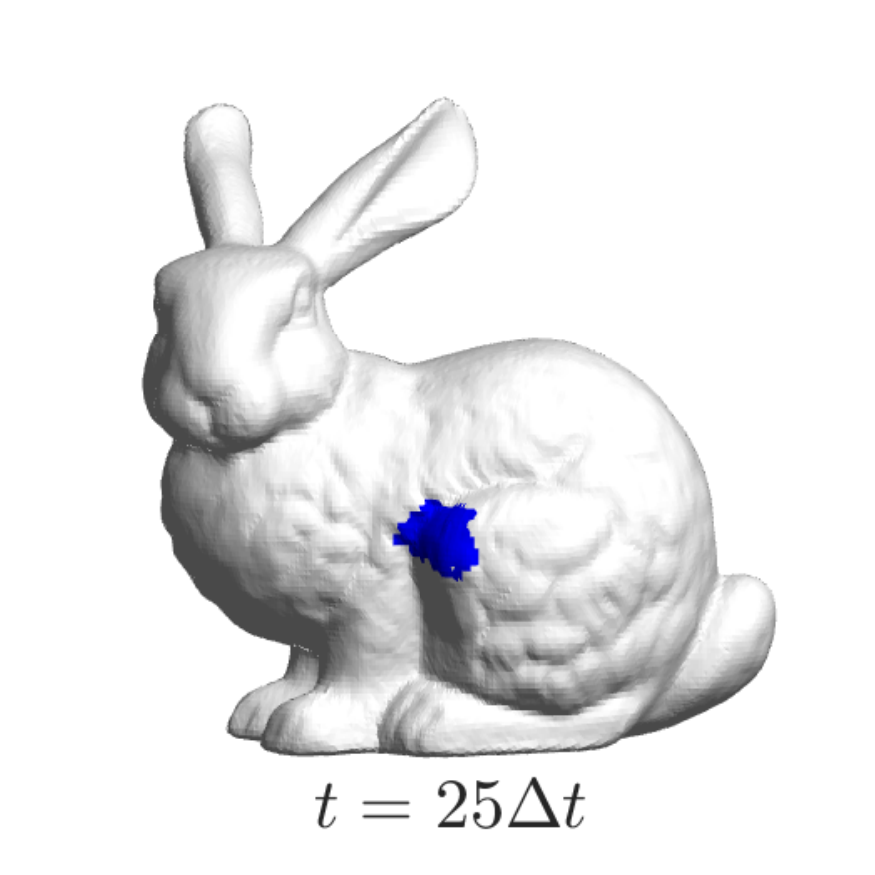}\\
    \includegraphics[width=1.5in]{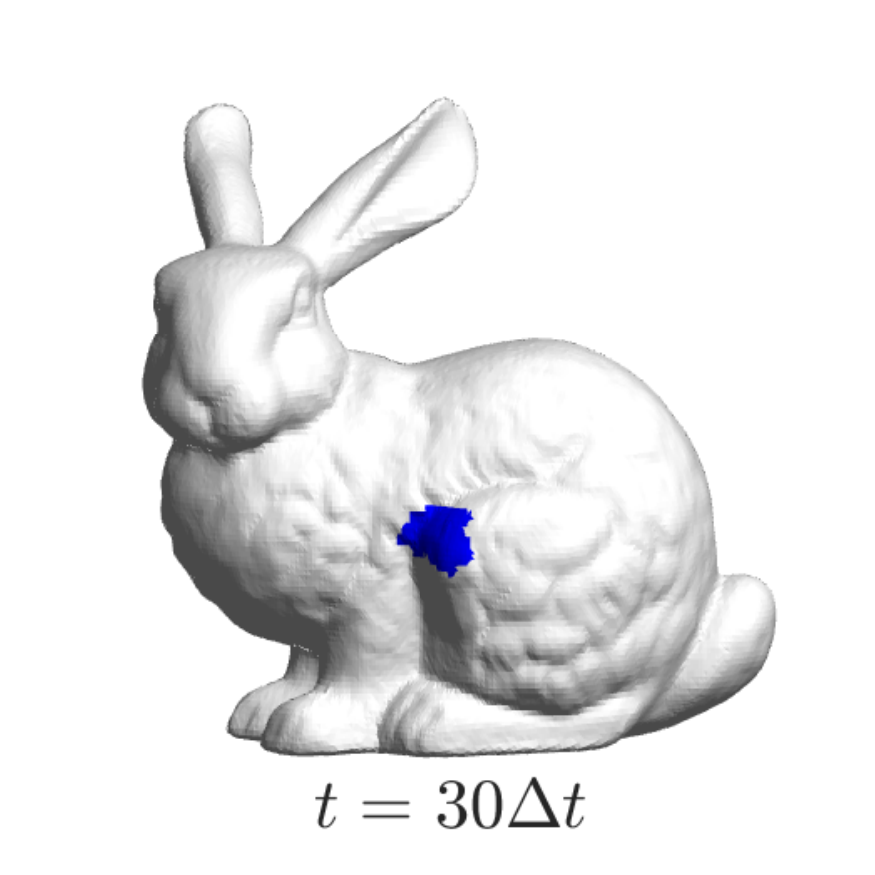}  & \includegraphics[width=1.5in]{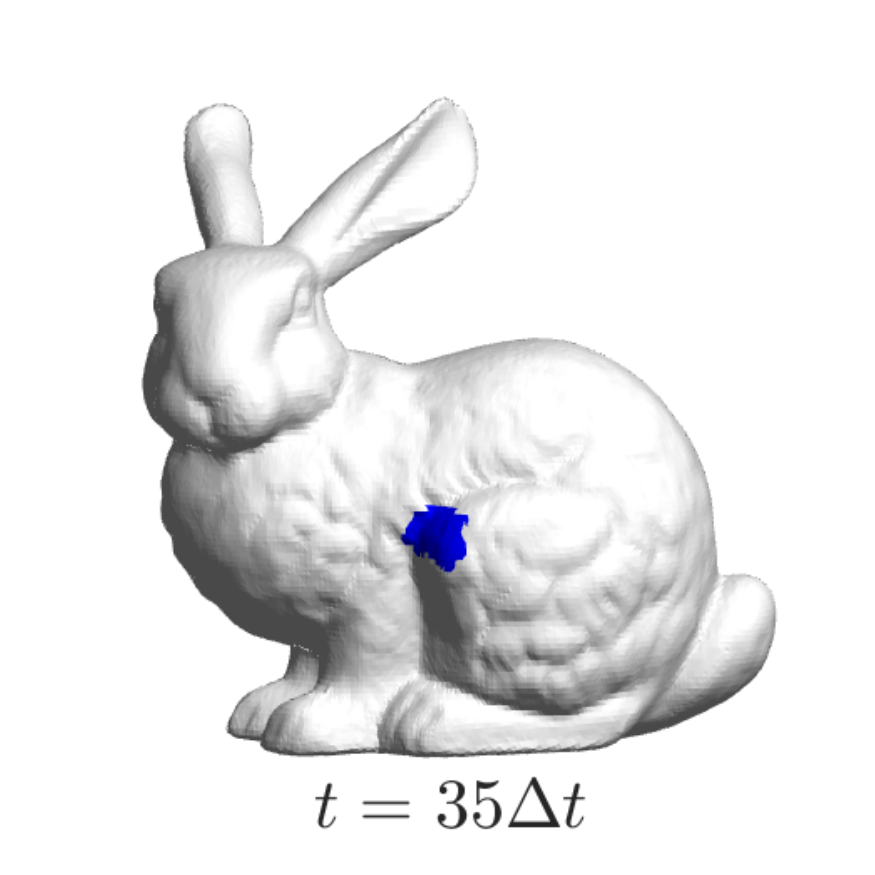} &
  \includegraphics[width=1.5in]{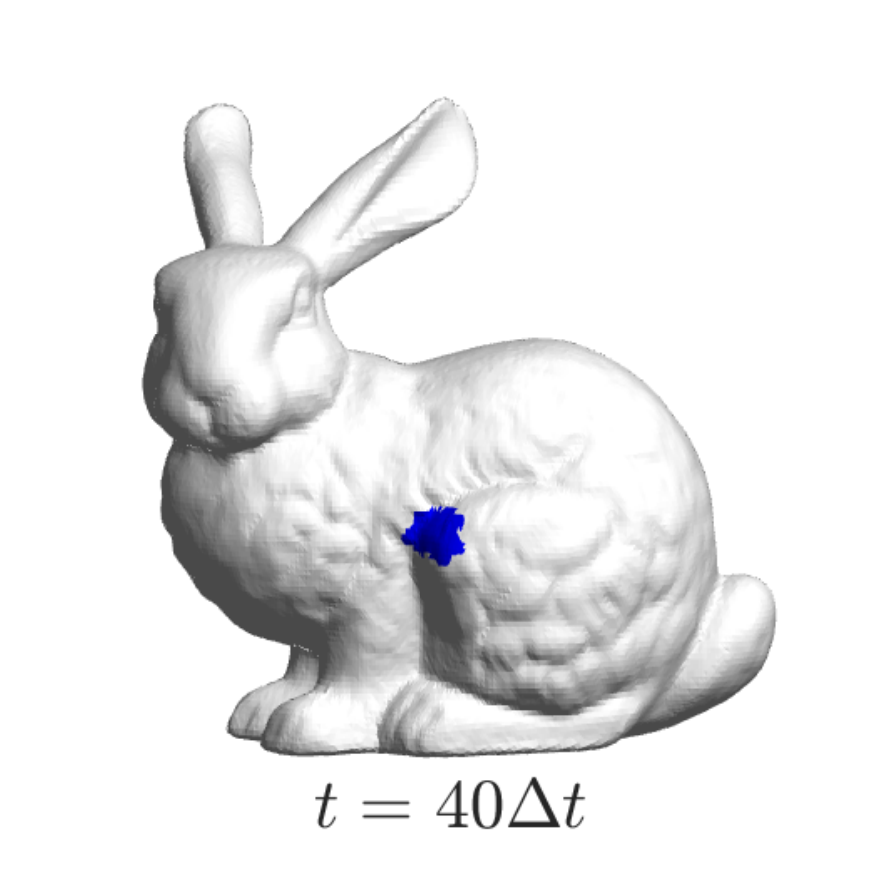}\\
      \includegraphics[width=1.5in]{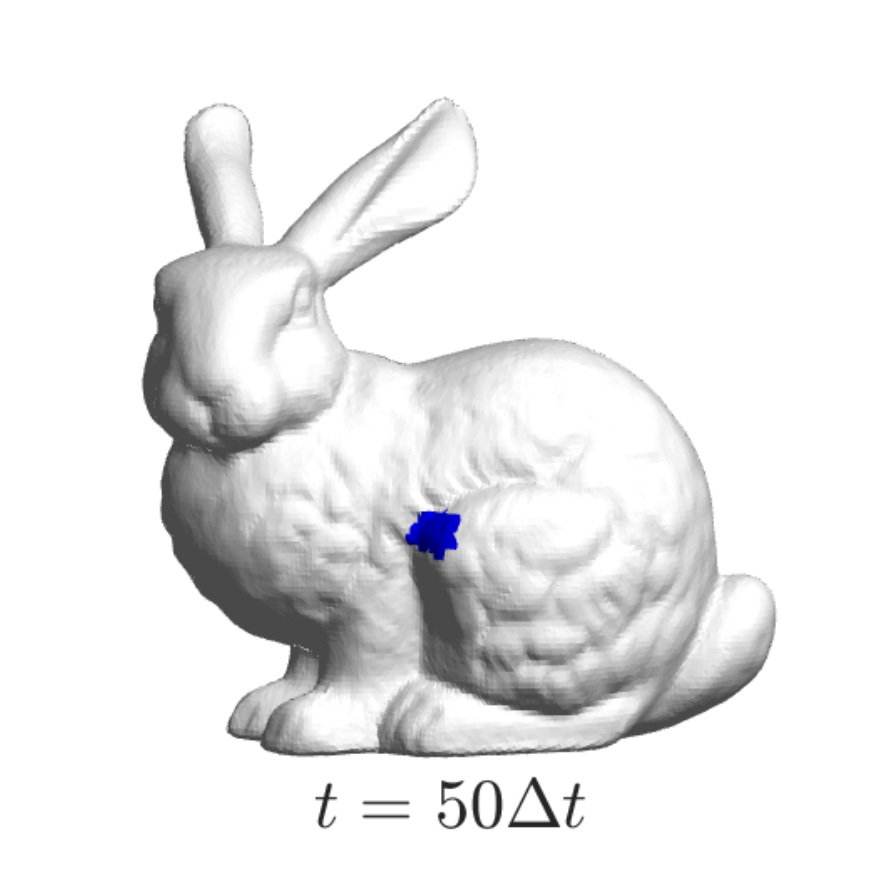}  & \includegraphics[width=1.5in]{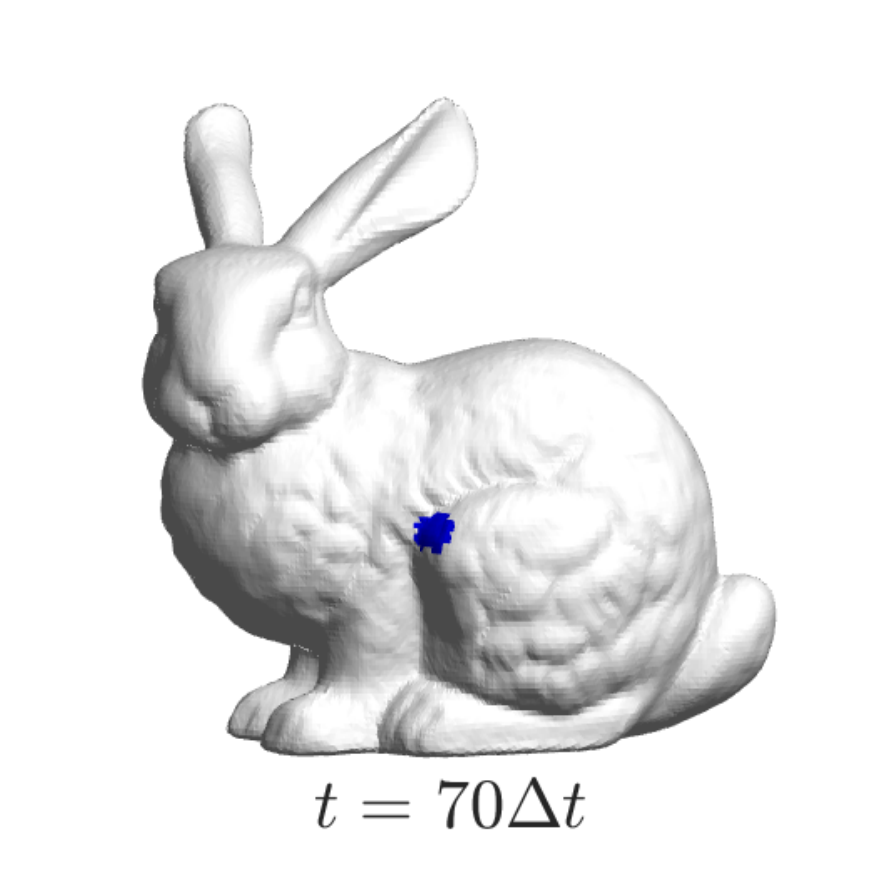} &
    \includegraphics[width=1.5in]{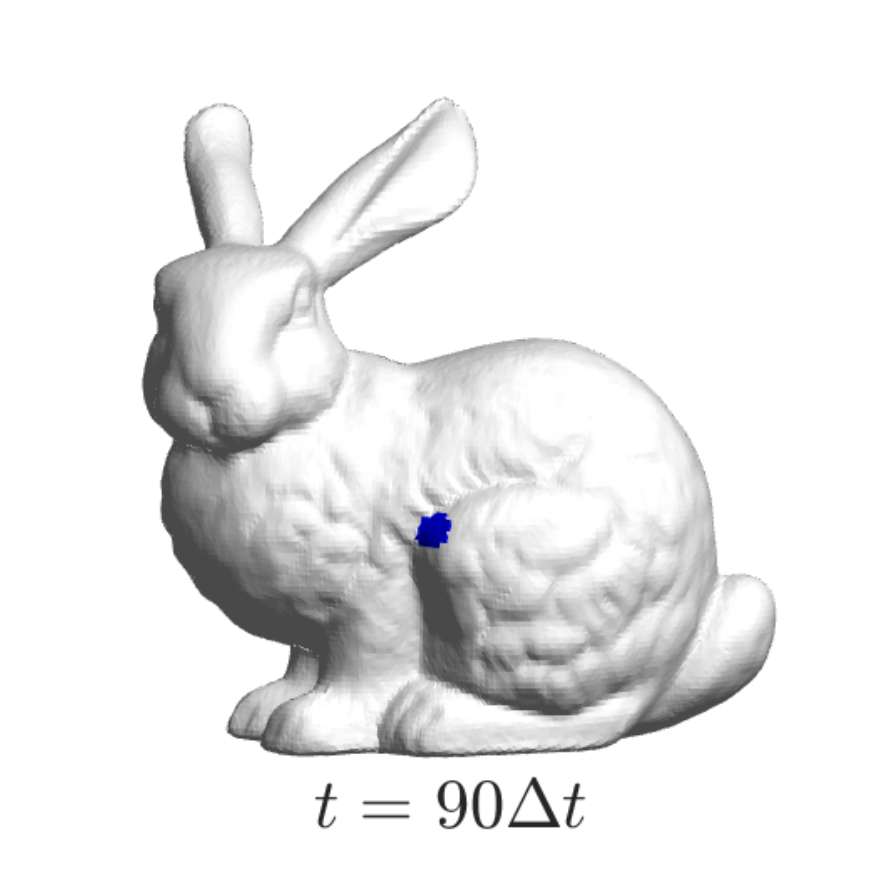}\\
\end{tabular}
\caption{The evolution of a random map from a torus to the Stanford bunny. The initial random map is shown in the top left corner in blue on the bunny. The map converges to a point.}
\label{fig:rand_map}
\end{figure}
To compute the harmonic map, we use second-order centred differences in space with a spatial step-size $\Delta x = 0.05$ in a banded computational domain $\Omega_c$ around the torus. To advance in time, forward Euler time-stepping with $\Delta t = 0.1 \Delta x^2$ is used.  The closest point function for the Stanford bunny is evaluated in the same way as our previous triangulated manifold example (see Section~\ref{sec:hand}). 

The $\text{CPM}_{\M}^{\N}$ evolution displayed in Figure~\ref{fig:rand_map} shows that the initial random map converges to a point (as expected). See Section 5.2 of~\cite{Memoli2004} for the corresponding $\text{LSM}_{\M}^{\N}$ calculation.

\subsection{Enhancing colour images via chromaticity diffusion}
\label{sec:colour_im_denoise}
We now consider colour image enhancement, a topic that can lead to harmonic maps and $p$-harmonic maps. One approach to remove noise from a colour image is to denoise the RGB-intensity values $I = [I_R,I_G,I_B]^T.$ However, colour artifacts are frequently observed with this approach. These artifacts are attributed to not preserving the direction of $I,$ which is called the chromaticity. For this reason, it is often preferred to denoise the intensity $I$ and the chromaticity $$\u =\frac{ I}{\|I\|_2},$$ separately~\cite{Tang2001,Vese2002}. 

The chromaticity is a map, $\u(\x): \M \rightarrow S^2,$ from a plane $\M\subset \mathbb{R}^2$ to the unit sphere $S^2,$ which can be denoised using the $\text{CPM}_{\M}^{\N}$. To illustrate, algorithms for $p$-harmonic maps with $p=2$ (isotropic diffusion) and $p=1$ (anisotropic diffusion) are implemented in this subsection. In our examples, noise is only added to the chromaticity of an image. This allows us to consider denoising by chroma diffusion without the added complexity of intensity diffusion. 

``Salt and pepper'' chromaticity noise is applied to the original image in the following manner. Some small subset of image pixels (5\% in our examples) is chosen in a uniformly random manner. At each randomly selected pixel, $\u(\x)$ is set to the direction of red, green or blue in a uniformly random way (e.g., set $\u(\x) = (1,0,0)^T$ if red). This gives the initial, noisy chromaticity map, $\u^0(\x).$ The intensity $I(\x)$ of the original image remains unchanged.

To denoise the chromaticity with isotropic diffusion we apply the $\text{CPM}_{\M}^{\N}$  for harmonic maps (Algorithm~\ref{alg:cpmmm} with $\mathbf{F} = \Delta_{\M} \u$). The flow is evolved until a visual inspection indicates the noise is sufficiently removed; stopping criteria based on reaching steady state could also be implemented. To avoid interpolation of the initial map $\u^0,$ we take $\Delta x = 1$  pixel. We apply second-order centred finite differences in space and forward Euler in time with $\Delta t = 0.1 \Delta x^2.$ Figure~\ref{fig:an-isotropic} shows the original noisy image (left) and the isotropically denoised result (middle) for a cartoon image of Newfoundland row houses~\cite{DougBird}. 
\begin{figure}
\centering
\begin{tabular}{ccc}
\includegraphics[width=2in]{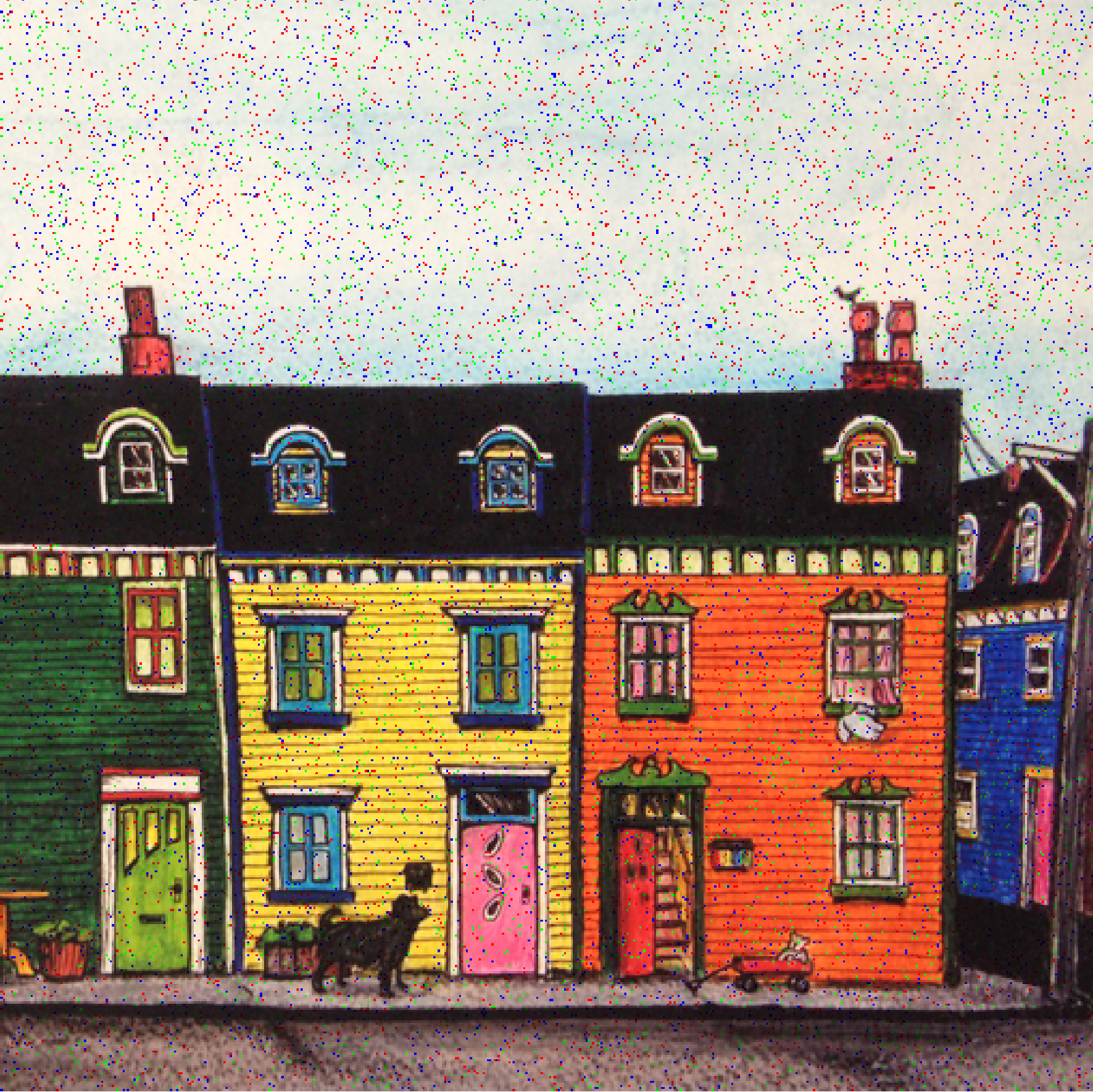} 
    &\includegraphics[width=2in]{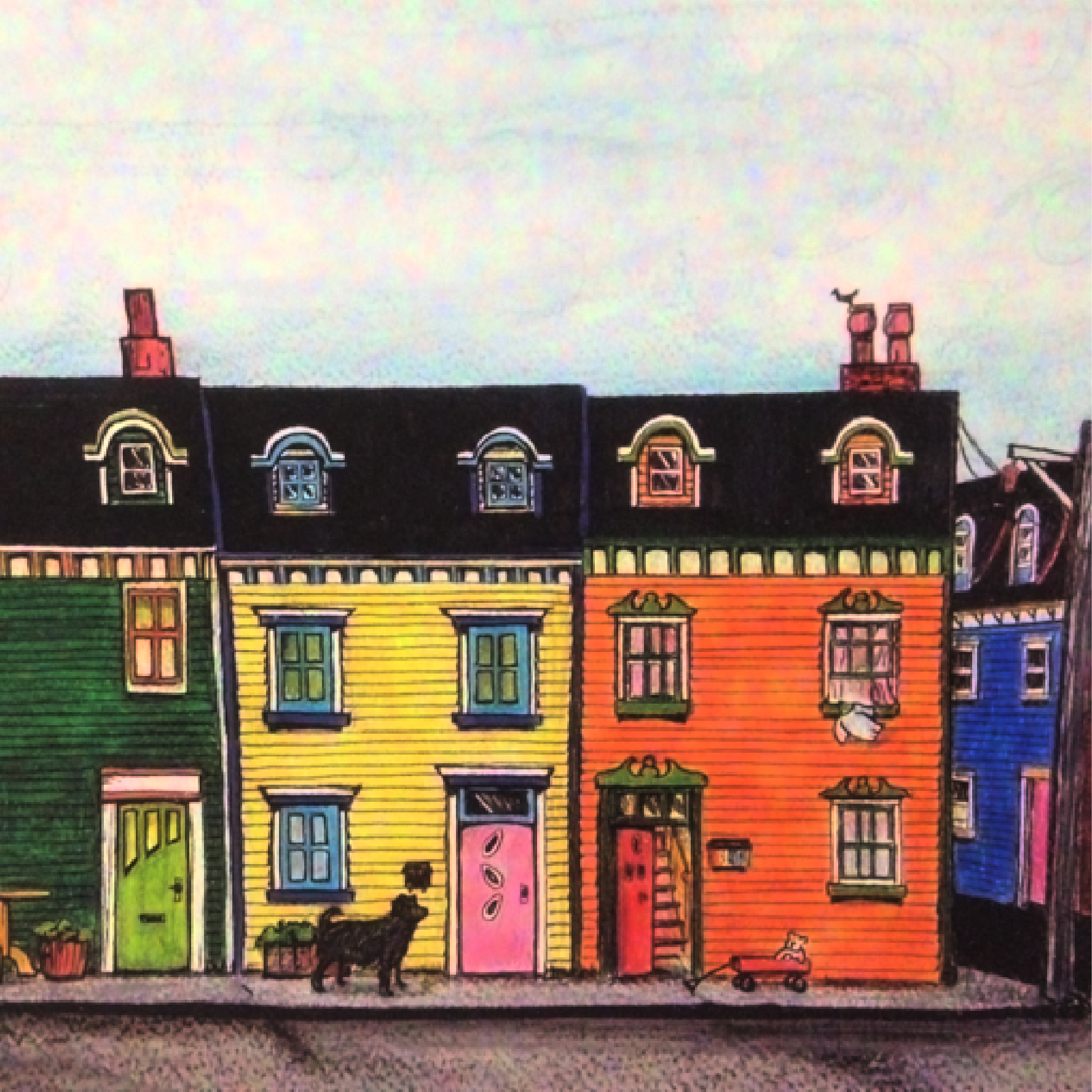} 
&\includegraphics[width=2in]{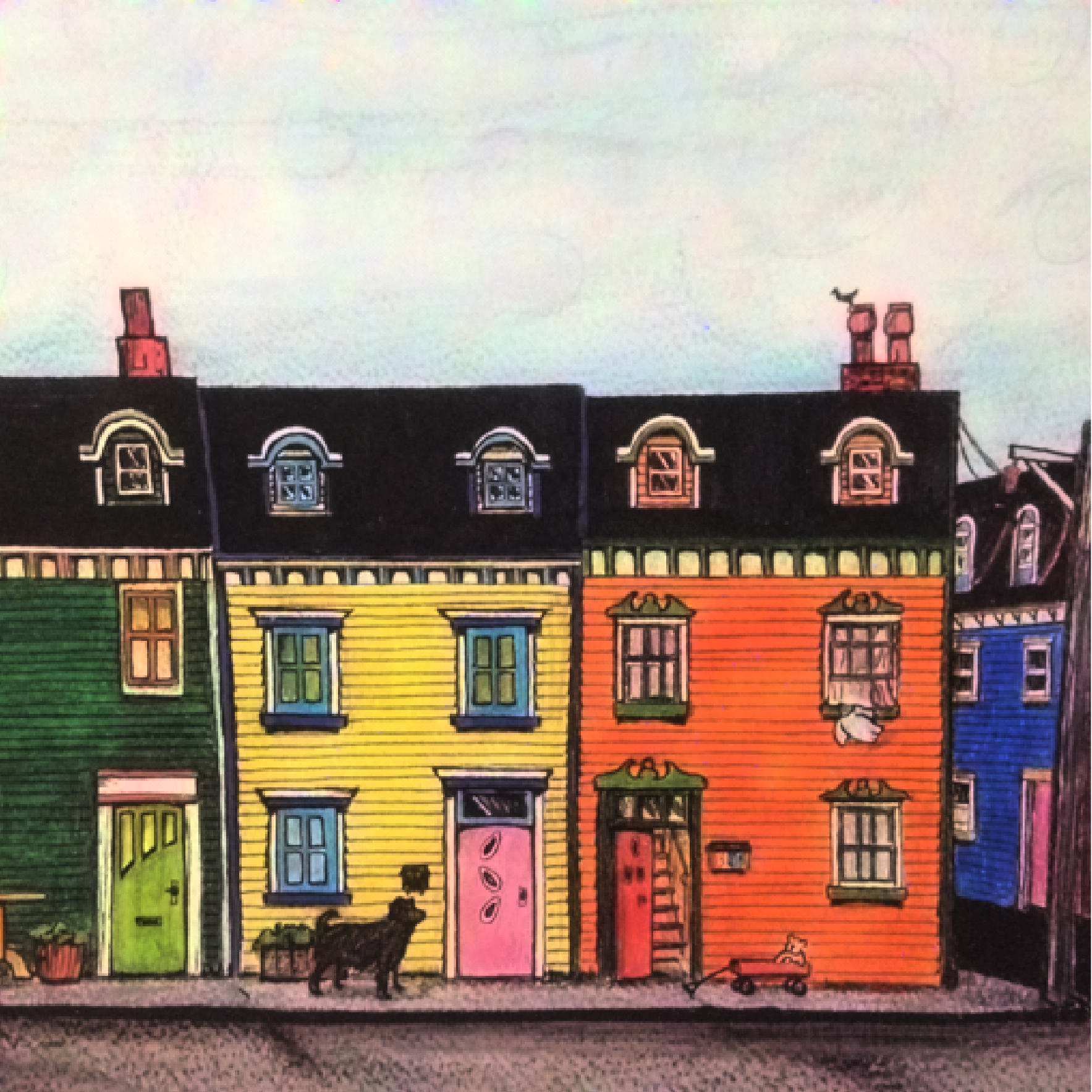} \\
\end{tabular}
\caption{Isotropic and anisotropic diffusion of chromaticity noise in a $512 \times 512$ pixel image. A cartoon of Newfoundland row houses~\cite{DougBird} with noise added to the chromaticity (left) was denoised with 40 times steps of isotropic diffusion (middle). Anisotropic diffusion applied for 120 time steps gives another denoised image (right).}
\label{fig:an-isotropic}
\end{figure}

Anisotropic chromaticity diffusion is slightly more involved numerically. The anisotropic diffusion of the initial, noisy map is carried out using~(\ref{pharm_proj_graddes}) with $p=1$ and $\M\subset \mathbb{R}^2,$ which simplifies to
 \begin{equation}
\left\{ \begin{aligned}
&\frac{\partial \u}{\partial t} = \Pi_{T_{\u} \N} \left(\nabla \cdot \left(\frac{ \J_{\u}}{\|\J_{\u}\|_{\mathcal{F}}}\right)\right),\\
&\u (\x,0) = \u_0 (\x),\\
&\J_{\u} \mathbf{n}|_{\partial \M} = 0.
\end{aligned} \right.
\label{anisotropic_diff}
\end{equation}
 As mentioned in Section~\ref{sec:cpmmm}, the PDE~(\ref{anisotropic_diff}) can be numerically approximated using the $\text{CPM}_{\M}^{\N}$ (Algorithm~\ref{alg:cpmmm}). Each row of $\J_{\u}$ is discretized using first-order forward finite differences in space. An approximation of the divergence of each row of $\J_{\u}/\|\J_{\u}\|_{\mathcal{F}}$ is then obtained using first-order backward finite differences. Forward Euler time-stepping is once again used, but with a time step-size of $\Delta t = 0.5\Delta x^2.$ We avoid division by zero by replacing the denominator with $\|\J_{\u}\|_{\mathcal{F}} + \delta,$ where $\delta \in \mathbb{R}$ is some small positive constant ($\delta = 10^{-16}$ here).

Figure~\ref{fig:an-isotropic} (right) shows the anisotropically denoised image. Both results in Figure~\ref{fig:an-isotropic} are good, and it is difficult to observe differences between isotropic and anisotropic diffusion. Figure~\ref{fig:three_colors} shows a clearer example of how anisotropic diffusion preserves the edges between different colours better than isotropic diffusion. 
\begin{figure}
\centering
\begin{tabular}{ccc}
\includegraphics[width=1.47in]{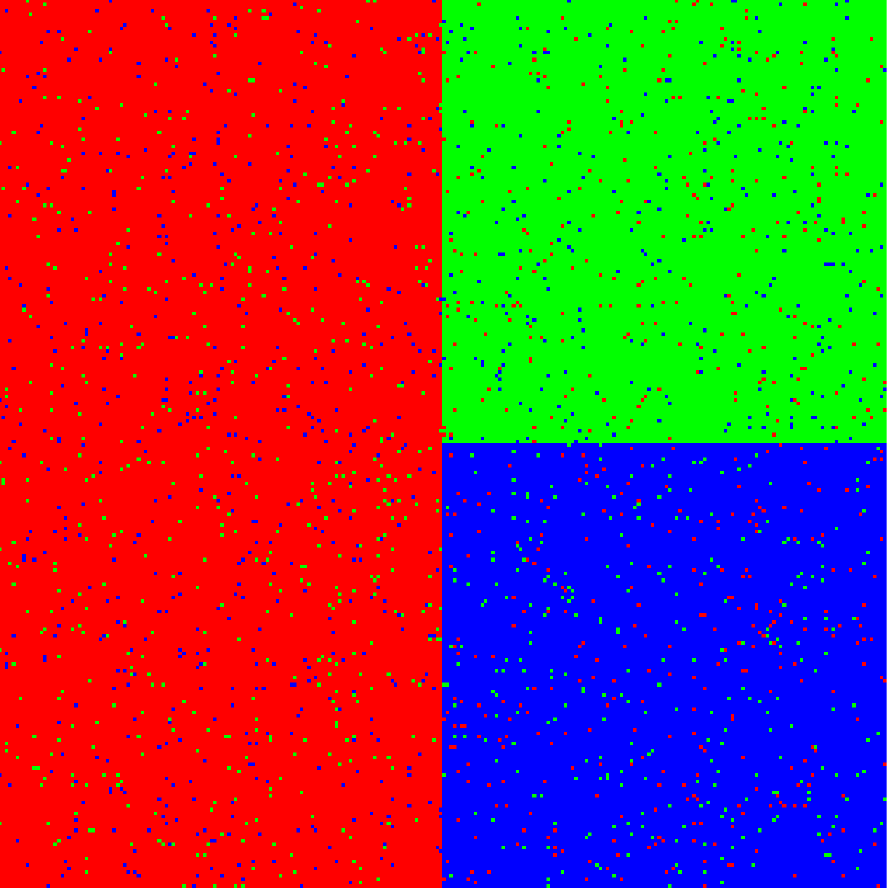}&
    \includegraphics[width=1.47in]{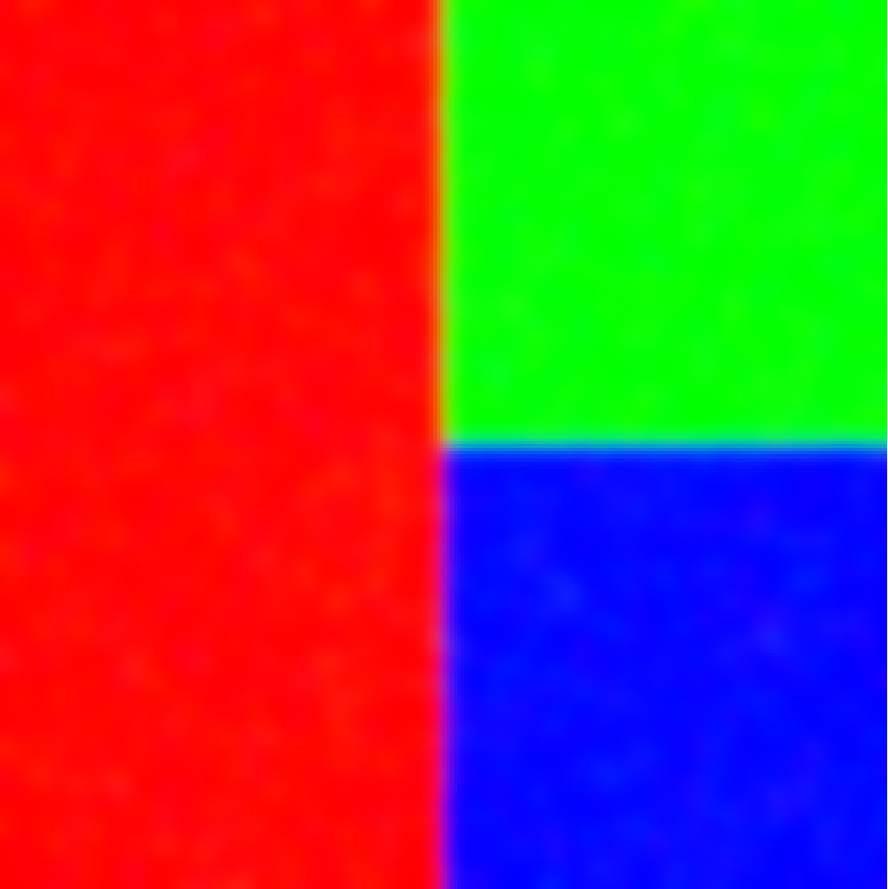}  & \includegraphics[width=1.47in]{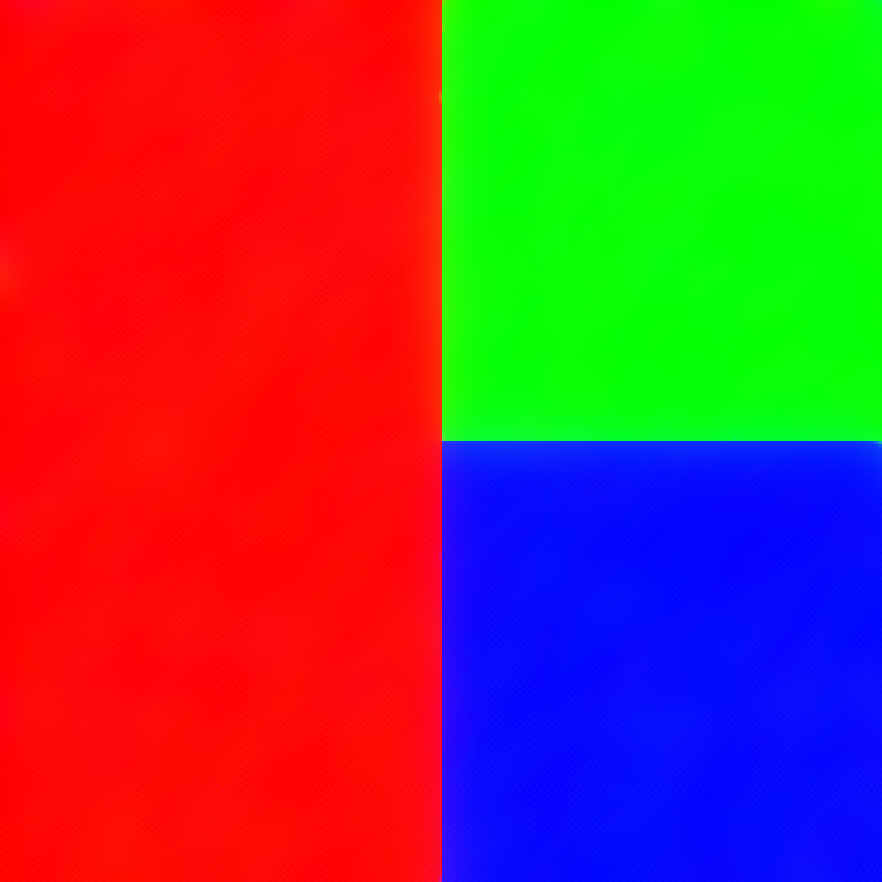} \\
\end{tabular}
\caption{Comparison of isotropic and anisotropic diffusion for a simple $256 \times 256$ pixel red, green and blue image. Chromaticity noise is added (left) and denoised with 30 time steps of isotropic diffusion (middle) and 125 time steps of anisotropic diffusion (right). }
\label{fig:three_colors}
\end{figure}
Edge blur arises between colours for isotropic diffusion, while anisotropic diffusion gives sharp edges.

\section{Conclusion}
\label{sec:conc}
This paper establishes a numerical framework for solving variational problems and PDEs that define maps from a source manifold $\M$ to a target manifold $\N.$ In our approach, the problem is embedded into the surrounding space by writing all geometric quantities intrinsic to $\M$ and $\N$ in terms of $\cp_{\M}$ and $\cp_{\N},$ respectively. The corresponding {\it closest point method for manifold mapping}, $\text{CPM}_{\M}^{\N},$ applies to a wide variety of variational problems and PDEs (see, e.g., (\ref{gen_proj})). Particularly, important cases that our work focuses on are the harmonic and $p$-harmonic maps.

For general mapping problems of the form~(\ref{gen_proj}), the $\text{CPM}_{\M}^{\N}$ (Algorithm~\ref{alg:cpmmm}) alternates between a step of the $\text{CPM}_{\M}$ for PDE evolution intrinsic to $\M$ and a projection step onto $\N$ using $\cp_{\N}.$  Splitting the evolution into two steps reduces the problem of solving a PDE with quantities on both $\M$ and $\N$ to the separate, simpler problems of solving a PDE on $\M$ alone and a projection onto $\N$ via $\cp_{\N}.$ It also eliminates the projection operator $\J_{\cp_{\N}},$ yielding additional computational savings. Consistency of the $\text{CPM}_{\M}^{\N}$ with the original constrained PDE  was shown in Theorem~\ref{theo1}.

Presently, the {\it level set method for manifold mapping}~\cite{Memoli2004}, $\text{LSM}_{\M}^{\N},$ is the most popular method for mapping between general manifolds. The $\text{CPM}_{\M}^{\N}$ is simpler and allows for more general manifold geometry than the $\text{LSM}_{\M}^{\N}.$ In practice, it also exhibits improved stability, computational speed, accuracy, and convergence rates. We illustrate the performance of our method on examples for denoising texture maps, diffusing random maps, and enhancing colour images. 

There are many interesting opportunities for future work.  Of particular interest is the development and study of methods for more general variational problems and PDEs. The study of applications is another rich subject for future work. Interesting examples include the texture mapping method of Dinh et al.~\cite{Dinh2005}, direct mapping of optic nerve heads~\cite{Gibson2010} and direct cortical mapping~\cite{Shi2007,Shi2009}. 

\section*{Acknowledgements}
This research was partially supported by an NSERC Canada grant (RGPIN 227823). The first author was also supported by an NSERC Alexander Graham Bell Canada Graduate Scholarship (Master's).

 \appendix
 \section{The Euler-Lagrange equations for liquid crystals}
 \label{app:liquid_cry}
In this appendix, we illustrate the derivation of the Euler-Lagrange equations~(\ref{cp_EL}) for the important case of liquid crystals, i.e., the case where $\M$ is a flat, open subset of $\mathbb{R}^m$ and $\N=S^{n-1}.$ 

Recall from~(\ref{cp_func}) that the closest point function can be written as 
\begin{equation*}
\cp_{\N}(\y) = \y - \d_{\N}(\y) \nabla \d_{\N}(\y).
\end{equation*}
The signed distance function for the unit hypersphere is
\begin{equation*}
\d_{\N}(\y) = \|\y\|_2 -1, \text{ for all } \y\in\mathbb{R}^n,
\end{equation*}
which gives $\cp_{\N}(\y) = \y/\|\y\|_2.$ Next, the gradient of the $i$-th component of $\cp_{\N}$ is 
\begin{equation*}
\begin{aligned}
\nabla \left\{\cp_{\N}(\y)\right\}_i &= \nabla \left(\frac{y_i}{\|\y\|_2}\right),\\
&= \frac{\nabla y_i}{\|\y\|_2} - y_i\frac{\y}{\|\y\|^3_2}.
\end{aligned}
\end{equation*}
Through a similar calculation
\begin{equation}
\left\{\H_{\cp_{\N}}^i(\y)\right\}_{jk} = 3 \frac{y_i y_j y_k}{\|\y\|_2^5} - \frac{y_k\delta_{ij} + y_j \delta_{ik} + y_i \delta_{jk}}{\|\y\|^3_2},
\label{hess_cp}
\end{equation}
where $\delta_{ij}$ is the Kronecker delta. 

The target manifold constraint that $\u \in S^{n-1}$ simplifies~(\ref{hess_cp}) since $\|\u\|_2 = 1.$ A further simplification is obtained using the identity  
\begin{equation}
 \nabla \d_{\N}(\u(\x)) \cdot\frac{ \partial \u(\x)}{\partial x_{\ell} } = \frac{\partial \d_{\N}(\u(\x))}{ \partial x_{\ell}} = 0,
\label{zero_dis_der}
\end{equation} 
which is derived by differentiating $\d_{\N}(\u(\x)) = 0$ for any $\x\in \M.$ For the signed distance function of the unit hypersphere we have $\nabla \d_{\N}(\u(\x)) \cdot \partial \u(\x) /\partial x_{\ell} = \sum_j u_j (\partial u_j/\partial x_{\ell}).$ Carrying out the matrix-vector multiplication in~(\ref{cp_EL}) yields
\begin{align*}
\left\{\sum_{\ell=1}^m \H_{\cp_{\N}(\u)} \left[\frac{\partial \u}{\partial \x_{\ell}},\frac{\partial \u}{\partial \x_{\ell}}\right]\right\}_i &= \sum_{\ell=1}^m \sum_{j=1}^n \sum_{k=1}^n \left(3 u_i u_j u_k - u_k\delta_{ij} - u_j \delta_{ik} - u_i \delta_{jk}\right) \frac{\partial u_j}{\partial x_{\ell}} \frac{\partial u_k}{\partial x_{\ell}},\\
&= -\sum_{\ell=1}^{m} \sum_{j=1}^n u_i \left(\frac{\partial u_j}{\partial x_{\ell}}\right)^2,\\
&= - u_i \|\J_{\u}\|^2_{\mathcal{F}}.
\end{align*}
Substituting into~(\ref{cp_EL}) gives the Euler-Lagrange equations for liquid crystals~\cite{Virga1994} $$ \Delta \u + \|\J_{\u}\|^2_{\mathcal{F}} \u = 0.$$

 \section*{References}

  \bibliographystyle{elsarticle-num} 
    \bibliography{manifold_mapping}


\end{document}